\documentclass[reqno,11pt,a4paper]{scrartcl}
\usepackage{amsmath,amssymb}
\usepackage{paralist}
\usepackage[colorlinks=true,pdfpagelabels]{hyperref}
\hypersetup{urlcolor=red, citecolor=blue}
\usepackage[top=1in,bottom=1in,left=1in,right=1in]{geometry}
\usepackage{amsmath,amsthm,amssymb}
\newtheorem{thm}{Theorem}[section]
\newtheorem{cor}[thm]{Corollary}
\newtheorem{lem}[thm]{Lemma}

\theoremstyle{definition}
\newtheorem{defn}{Definition}[section]
\newtheorem{rem}[thm]{Remark}
\numberwithin{equation}{section}

\newcommand{\set}[1]{\left\{#1\right\}}

\newcommand{\eps}{\varepsilon}

\newcommand{\Om}{\Omega}
\newcommand{\Ombar}{\overline{\Om}}
\newcommand{\R}{\mathbb{R}}
\newcommand{\N}{\mathbb{N}}
\newcommand{\io}{\int_\Omega}
\newcommand{\intdom}{\int_{\partial\Omega}}
\newcommand{\intnt}{\int_0^t}
\newcommand{\norm}[2][]{\left\Vert#2\right\Vert_{#1}}
\newcommand{\Tmax}{T_{max}}
\newcommand{\na}{\nabla}
\newcommand{\phii}{\varphi}
\newcommand{\ddt}{\frac{d}{dt}}
\newcommand{\ubar}{\overline{u}}
\newcommand{\ue}{u_\eps}
\newcommand{\ve}{v_\eps}
\newcommand{\we}{w_\eps}
\newcommand{\De}{D_\eps}
\newcommand{\intnT}{\int_0^T}
\newcommand{\wstarto}{\overset{*}{\rightharpoonup}}

\usepackage{color}

\newcommand{\systemref}{\eqref{ueq}-\eqref{ic}}
\newcommand{\nn}{\nonumber}
\begin{document}
\date{\today}

\title
      {Boundedness in a chemotaxis-haptotaxis model with nonlinear diffusion}%

\author{Yan Li\thanks{Department of Mathematics, Southeast University, Nanjing 21189, PR China; Institut f\"ur Mathematik, Universit\"at Paderborn, Warburger Str. 100, 33098 Paderborn, Germany; email: \mbox{yezi.seu@gmail.com}} \ and  Johannes Lankeit\thanks{Institut f\"ur Mathematik, Universit\"at Paderborn, Warburger Str. 100, 33098 Paderborn, Germany; email: \mbox{johannes.lankeit@math.upb.de}}}



\maketitle
\begin{abstract}
This article deals with an initial-boundary value problem for the coupled chemotaxis-haptotaxis system with nonlinear diffusion
\begin{eqnarray*}
\left\{\begin{array}{lll}
     \medskip
     u_t=\nabla\cdot(D(u)\nabla u)-\chi\nabla\cdot(u\nabla v)-\xi\nabla\cdot(u\nabla w)+\mu u(1-u-w),\ \ \ &x\in \Omega,\ t>0,\\
      \medskip
     v_t=\Delta v-v+u,\ \ &x\in \Omega,\ t>0,\\
      \medskip
     w_t=-vw\ \ &x\in \Omega,\ t>0,
 \end{array}\right.
\end{eqnarray*}
under homogeneous Neumann boundary conditions in a bounded smooth domain $\Omega\subset\mathbb{R}^n$, $n=2, 3, 4$, where $\chi, \xi$ and $\mu$ are given nonnegative parameters.  The diffusivity $D(u)$ is assumed to satisfy $D(u)\geq\delta u^{m-1}$ for all $u>0$ with some $\delta>0$.  
It is proved that for sufficiently regular initial data global bounded solutions exist whenever $m>2-\frac{2}{n}$.
For the case of non-degenerate diffusion (i.e. $D(0)>0$) the solutions are classical; for the case of possibly degenerate diffusion ($D(0)\geq 0$), the existence of bounded weak solutions is shown.\\[0.2cm]
{\bf Math Subject Classification (2010):} 35K55 (primary), 35B40, 92C17\\
{\bf Keywords:} global existence, boundedness, nonlinear diffusion, chemotaxis, haptotaxis
\end{abstract}

\section{Introduction}
Already in early stages of cancer, malignant tumours may possess the ability to invade tissue in their neighbourhood, with harmful effects on the organism (\cite{liotta2000cancer}). 
Among the many mathematical models that have been developed for the description of the progress of cancer in different stages (see, for instance,   \cite{bellomo2008,chaplain2006,lowengrub2010nonlinear,gatenby_gawlinski,perumpanani_byrne,chaplain_anderson} and the refercences therein), in \cite{chaplain2006} Chaplain and Lolas introduced the following chemotaxis-haptotaxis system as a model describing the process of cancer invasion:
 \begin{eqnarray}\label{cthtfull}
\left\{\begin{array}{lll}
     \medskip
     u_t=D\Delta u-\nabla\cdot(\chi u\nabla v)-\nabla\cdot(\xi u\nabla w)+\mu u(1-u-w),\ \ \ &x\in \Omega,\ t>0,\\
      \medskip
     \tau v_t=\Delta v-v+u,\ \ &x\in \Omega,\ t>0,\\
      \medskip
     w_t=-vw+\eta w(1-u-w),\ \ &x\in \Omega,\ t>0.
 \end{array}\right.
\end{eqnarray}
Herein, $D, \chi, \xi, \mu>0$, $\eta\ge 0$ and $\tau\in\{0,1\}$ are given parameters and $u$ denotes the density of cancer cells, $v$ represents the concentration of so-called matrix degrading enzymes (MDEs) and $w$ is used to describe the non-diffusive concentration of the extracellular matrix (ECM). In addition to their undirected random movement, cancer cells are attracted by MDE and macromolecules that are bound in the ECM. Accordingly, their motion is biased toward higher concentrations of these substances with this process being known as chemotaxis, if the attractant is diffusible, and as haptotaxis for attractants fixed in the tissue. In addition, cancerous cells reproduce and compete amongst themselves and with healthy cells for space and nutrients, so that source terms of logistic type arise in the first equation. Matrix degrading enzymes, which are produced by cancer cells, diffuse, decay, and decompose the tissue around them for more living space. A choice of $\eta>0$ embodies the ability of the
ECM to remodel back to a normal level. However, since degradation occurs much faster than the re-establishment of the tissue, a choice of $\eta=0$ seems justified, and, indeed, is taken in many of the mathematical works concerning systems like \eqref{cthtfull}, including the present one. 

If all effects of the extracellular matrix are ignored by letting $w\equiv 0$ and if additionally $\mu=0$, \eqref{cthtfull} turns into the famous Keller-Segel model \cite{Keller-Segel-70} for chemotaxis, which has brought forth a large amount of mathematical literature during the last decades. We refer to the survey articles \cite{horstmann_I,hillen_painter,BBTW}. In this context one point of particular interest is the question whether solutions exist that blow-up in finite time or whether all solutions are global and, maybe, even bounded. And in fact, blow-up is known to occur and, according to more recent results, to be a generic phenomenon if $n>2$ or if $n=2$ and the initial mass $\io u_0$ is sufficiently large (\cite{jaeger_luckhaus_92,herrero_velazquez_97,winkler_11_blowuphigherdim,mizoguchi_winkler_13}).

The presence of logistic terms ($\mu>0$, but still $w\equiv 0$) not only leads to colourful dynamics (cf. \cite{kuto_osaki_sakurai_tsujikawa_12,painter_hillen_11,osaki2002exponential}), but presumably supports the global solvability as well. Nevertheless, proofs of this belief, 
both for the parabolic-elliptic ($\tau=0$) variant \cite{tello2007} and for the parabolic-parabolic ($\tau=1$) counterpart \cite{Winkler-10a}, require the restriction to the two-dimensional setting or that of sufficiently large values of $\mu>0$.
Although there are some closely related models where solutions are global  (cf. \cite{baghaei_hesaaraki_13,xinru_14,tao_11,tao_winkler_12,2015arXiv150302387Y,xiang2015boundedness}) and although global weak solutions are known to exist (\cite{lankeit2015eventual}), which, under certain conditions on the parameters, become smooth after some time (\cite{lankeit2015eventual}), it is still unknown whether finite-time blow-up may occur if $n\ge 3$ and $\mu>0$ is small.

Returning to models featuring a nontrivial third component, let us first report on
some results on haptotaxis-only models, i.e. models with $\chi=0$. For $\chi=\mu=\eta=0$, in \cite{morales_rodrigo08}, the local existence and uniqueness of classical solutions have been shown.
In \cite{marciniakczochraptashnyk}, a similar model (with nonlinear kinetics of the ECM, in that $+u$ in the second equation was replaced by $+uw$ \cite{perumpanani_byrne}) with $\chi=0=\eta$ but $\mu>0$ was proven to have global bounded weak solutions; the existence of a global classical solution was established in \cite{tao_zhu}. Asymptotic properties of solutions were investigated in \cite{lictcanu2010asymptotic}. For another related model (see \cite{anderson}), which additionally involves an equation governing the evolution of oxygen or other nutrients, in \cite{WalkerWebb} global existence was shown.

Also for a haptotaxis system including the effect of remodeling tissue ($\eta>0$) global classical solutions are known to exist, at least under the condition that the logistic growth is strong if compared to the effects of haptotaxis and regrowth, more precisely, if $\mu>\xi\eta$ \cite{tao_remodeling,fan_zhao}. 

As to models involving both chemotactic and haptotactic effects, there have been results guaranteeing global existence of classical solutions for $\eta=0$ under the condition that the chemotactic and haptotactic densities depend on the value of $u$ in such a way that $\chi(u)u$ and $\xi(u)u$ are bounded \cite{tao_cui,fan_zhao}.

For system \eqref{cthtfull} with instantaneous diffusion of enzymes and without tissue remodeling (i.e. $\tau=0=\eta$), Tao and Wang \cite{tao_wang_parell} proved global existence of solutions for $n=2$ or $n=3$ and sufficiently large $\mu$. Tao and Winkler \cite{taowk_bdnessandstabil} gave an explicit condition on $\mu$ to ensure global existence and boundedness and investigated attraction of solutions to the steady state $(1,1,0)$. In \cite{tao2014dominance}, they improved the condition on $\mu$, so that it coincides with the best one known for the parabolic-elliptic Keller-Segel system with logistic source (\cite{tello2007}), and also gave an explicit smallness condition on $w_0$, under which $w$ asymptotically becomes negligible.
Existence and uniqueness of a global classical solution to \eqref{cthtfull} with $\eta>0$ and $\tau=0$ have been obtained in \cite{tao2014energy} in the spatially two-dimensional setting.

The existence of global solutions to the fully parabolic system \eqref{cthtfull} (that is, $\tau=1$) has been established in \cite{tao2008global} under the condition of $\eta=0$ and $\mu$ being sufficiently large as compared to $\chi$.
Recently, Tao \cite{tao2014boundedness} and Cao \cite{cao2015boundedness} proved the boundedness of solutions in two- or three-dimensional domains, respectively.

The models mentioned above described the random part of the motion of cancer cells by linear diffusion. As pointed out in \cite[Section 6]{szymanska2009mathematical}, more appropriately from a physical point-of-view, however, it might be considered as movement in a porous medium. We will therefore assume $D$ to be a nonlinear function of $u$ and shall deal with
the following chemotaxis-haptotaxis system
\begin{align}
     u_t=&\nabla\cdot(D(u)\nabla u)-\chi\nabla\cdot(u\nabla v)-\xi\nabla\cdot(u\nabla w)+\mu u(1-u-w),\ \ \ &x\in \Omega,\ t>0,\label{ueq}\\
     v_t=&\Delta v-v+u,\ \ &x\in \Omega,\ t>0,\label{veq}\\
     w_t=&-vw,\ \ &x\in \Omega,\ t>0,\label{weq}\\
      \frac{\partial u}{\partial \nu}&=\frac{\partial w}{\partial \nu}=\frac{\partial v}{\partial \nu}=0,\ \ &x\in\partial\Omega,\ t>0,\label{rb}\\
      u(x, 0)=&u_0(x),\ v(x, 0)=v_0(x),\ w(x, 0)=w_0(x)\ \ &x\in\Omega\label{ic}
\end{align}
in a bounded domain $\Omega\subset \mathbb{R}^n$ with smooth boundary, $n\in\{2, 3, 4\}$.  The functions $u_0, v_0, w_0$ are supposed to satisfy the smoothness assumptions
\begin{eqnarray}\label{eq:initdatacond}
\left\{\begin{array}{lll}
     \medskip
     u_0\in C^0(\bar{\Omega}) \mbox{ with } u_0\geq0 \mbox{ in }\Omega\; \mbox{ and }\ u_0\not\equiv0,\\
      \medskip
     v_0\in\ W^{1, \infty}(\Omega) \mbox{ with }v_0\geq0\ \mbox{ in } \Omega,\\
      \medskip
     w_0\in\ C^{2+\alpha}(\bar{\Omega}) \mbox{ for some }\alpha>0 \mbox{ with } w_0>0 \mbox{ in }\ \bar{\Omega} \mbox{ and } \frac{\partial w_0}{\partial\nu}=0\ \mbox{on } \partial\Omega.
 \end{array}\right.
\end{eqnarray}
Due to the condition $\frac{\partial w_0}{\partial\nu}=0$ it is already ensured that $\frac{\partial w}{\partial \nu}=0$ on $\partial \Om$ and for $t>0$, so that \eqref{rb} is equivalent to
\[
 \frac{\partial u}{\partial \nu}-\chi u\frac{\partial v}{\partial \nu}-\xi u\frac{\partial w}{\partial \nu}=\frac{\partial v}{\partial \nu}=0\qquad x\in\partial\Om,t>0.
\]
We furthermore assume that 
\begin{eqnarray}\label{D12}
D\in C^2([0, \infty)) \quad \mbox{ and }\quad D(u)\geq \delta u^{m-1}\;\; \mbox{for all } u>0
\end{eqnarray}
with some $\delta>0$ and $m>1$ which will be subject to additional assumptions. In addition, if $D(u)$ satisfies
\begin{eqnarray}\label{D:nondeg}
D(0)>0,
\end{eqnarray}
the diffusion is non-degenerate and one may hope for classical solutions.

For certain values of $m$, in \cite{tao2008global} global existence of classical solutions to \systemref\ was established under these conditions.
Nonetheless, the question of boundedness of solutions to system \systemref\ is still open.
Thus it is meaningful to analyze the following question:
\begin{equation}\label{Q1}\tag{Q1}
\mbox{Which\ size\ of\ }m\mbox{\ is\ sufficient\ to\ ensure\ boundedness\ of\ solutions\ to\ \systemref?}
\end{equation}

Variants of the quasilinear Keller-Segel system with logistic source have been studied in \cite{li2015boundedness,xinru_14,wang2014boundedness,wang2014quasilinear,cao_zheng,zheng} for instance.
Inter alia, Li and Xiang \cite{li2015boundedness} showed global existence of bounded solutions to a fully parabolic quasilinear Keller-Segel system with logistic source and degenerate diffusion, that is, without assumption \eqref{D:nondeg}.
Hence another natural question to ask is:
\begin{equation}\label{Q2}\tag{Q2}
\mbox{Do bounded solutions to \systemref\ with degenerate diffusion exist?}
\end{equation}

It is our goal in this work to give answers to (Q1) and (Q2).  Namely, for non-degenerate and degenerate diffusion both, we will show the existence of global-in-time solutions to \systemref\ that are uniformly bounded. Our main results read as follows:

\begin{thm}\label{Theorem1}
Let $n\in\{2, 3, 4\}$ and $\Omega\subset\mathbb{R}^n$ be a bounded domain with smooth boundary and $\chi, \xi, \mu>0$.
Suppose that the initial data $(u_0, v_0, w_0)$ satisfy $(\ref{eq:initdatacond})$. Then for any $\delta>0$ and
\begin{eqnarray}\label{m}
m>2-\frac{2}{n}
\end{eqnarray}
there is $C>0$ such that for any function $D$ satisfying \eqref{D12} and \eqref{D:nondeg}, system \systemref\ has a classical solution $(u,v,w)\in\left(C^0(\bar{\Omega}\times[0,\infty)\cap C^{2, 1}(\Ombar\times(0,\infty)))\right)^3$ which exists globally in time and satisfies
\begin{equation}\label{eq:defbdness}
\norm[L^\infty(\Om)]{u(\cdot,t)}+\norm[W^{1,\infty}(\Om)]{v(\cdot,t)}+\norm[L^\infty(\Om)]{w(\cdot,t)}\leq C
\end{equation}
for all $t\in(0,\infty)$.
\end{thm}
In the case of possibly degenerate diffusion, that is without the assumption (\ref{D:nondeg}), $\systemref$ possesses at least one global bounded weak solution:
\begin{thm}\label{Theorem2}
Let $n\in\{2, 3, 4\}$ and $\Omega\subset\mathbb{R}^n$ be a bounded domain with smooth boundary and $\chi, \xi, \mu>0$.
Suppose that the initial data $(u_0, v_0, w_0)$ satisfy $(\ref{eq:initdatacond})$. Then for any $\delta>0$ and
\(
m>2-\frac{2}{n}
\)
for any function $D$ satisfying \eqref{D12}, system \systemref\ has a weak solution $(u,v,w)$ in the sense of Definition \ref{def:weaksol} below that exists globally in time and is bounded in the sense that \eqref{eq:defbdness} holds.
\end{thm}

\begin{rem}
In the case $n=2$, condition \eqref{m} reads $m>1$ and coincides with the condition posed in order to obtain the global existence result in  \cite{tao2011chemotaxis}. For $n=3$, \eqref{m} is stronger than the condition $m>\frac{26}{21}$ in \cite{tao2011chemotaxis} and we leave open the question here whether the solutions to \systemref\ for $\frac{26}{21}<m<\frac{4}{3}$ are bounded.
\end{rem}

\begin{rem}
 The chemotaxis-haptotaxis system therefore has bounded solutions under the same condition on $m$ as the pure chemotaxis system with $w\equiv 0$ without logistic source \cite{Tao-Winkler-jde}. For $\mu=0$ this condition is essentially optimal, cf. \cite{Winkler_volumefilling}.
\end{rem}

The rest of this paper is organized as follows. In the following section, we will prepare the later proofs by collecting some useful estimates. Section \ref{s3} is devoted to deriving a differential inequality for the quantity $\io u^p+\io |\na v|^{2q}$, $p,q>1$, whose boundedness is asserted in Lemma \ref{l10} and which can be used to prove the boundedness of solutions and thus Theorem  \ref{Theorem1}. Relying on the existence of classical solutions in the non-degenerate case, in Section \ref{s4} we will then complete the proof of Theorem \ref{Theorem2} by an approximation procedure.

\section{Preliminaries}\label{s2}
In this section we collect a few short, but helpful results, most of which have already been proven elsewhere.
In some cases, for the convenience of the reader, we take the liberty of stating the conditions differently from the original source, in such a way that the lemmata are still covered by the original proof, but become more easily applicable in the present situation.
For example, Lemma \ref{lem:bdDeltaw} below was proven in \cite{tao2014boundedness} for a system of PDEs differing from \systemref\ by the first equation.
The only property of $u$ required in the proof, however, is nonnegativity, so that this difference plays no role, and we state the lemma accordingly.

Let us begin with an estimate for a particular boundary integral that enables us to cover possibly non-convex domains.

\begin{lem}\label{lem:bdryestimate}
 Let $\Om\subset\R^n$ be a bounded domain with smooth boundary, let $q\in [1,\infty)$ and $M>0$. Then for any $\eta>0$ there is $C_\eta>0$ such that for any $v\in C^{2}(\Ombar)$ satisfying $\frac{\partial v}{\partial \nu} = 0$ on $\partial\Om$ and $\io |\na v|\leq M$ the inequality
\[
 \intdom |\na v|^{2q-2} \frac{\partial |\na v|^2}{\partial \nu} \leq \eta \io |\na |\na v|^q|^2+ C_\eta
\]
holds.
\end{lem}
\begin{proof}
 This has been proven in the course of the proof of \cite[Prop. 3.2]{ishida2014boundedness} by a combination of an estimate of $\frac{\partial |\na v|^2}{\partial\nu}$ on the boundary in terms of the curvature of $\partial\Om$ and $|\na v|^2$ with the embedding $W^{r+\frac12,2}(\Om)\hookrightarrow W^{r,2}(\partial \Om)\hookrightarrow L^2(\partial \Om)$, a fractional Gagliardo-Nirenberg inequality and Young's inequality (cf. Lemmata 2.2 -- 2.5 and estimate (3.10) of \cite{ishida2014boundedness}).
\end{proof}

We will make use of the Poincar\'e inequality in the following form:
\begin{lem}\label{lem:poincare}
 Let $\Om\subset\R^n$ be a bounded smooth domain. Let $\alpha>0$. Then there exists $C>0$ such that
\[
 \norm[W^{1,2}(\Om)]{u} \leq C \left( \norm[L^2(\Om)]{\na u} + \left(\io |u|^\alpha \right)^{\frac1\alpha}\right) \qquad \mbox{for all } u\in W^{1,2}(\Om).
\]
\end{lem}
\begin{proof}
 Assuming to the contrary that there exists a sequence $(v_n)_{n\in\N}\subset W^{1,2}(\Om)$ such that $\norm[L^2(\Om)]{v_n}=1$ and $\norm[L^2(\Om)]{v_n}\geq n(\norm[L^2(\Om)]{\na v_n} + (\io |v_n|^\alpha)^{\frac1\alpha})$ for any $n\in\mathbb{N}$, we can find a subsequence thereof converging to some $v\in W^{1,2}(\Om)$ weakly in $W^{1,2}(\Om)$ and strongly in $L^2(\Om)$ so that $\norm[L^2(\Om)]{v}=1$. Due to $\io |v|^\alpha \leq \liminf_{n\to \infty} \io |v_n|^\alpha \leq \liminf_{n\to\infty} (\frac1n)^\alpha=0$, at the same time we obtain $v=0$, a contradiction.
\end{proof}

Also the Gagliardo-Nirenberg inequality will be used in a less common version (for a similar variant see \cite[Lemma 3.2]{wk_criticalexponent}):
\begin{lem}
 Let $\Om\subset \R^n$ be a bounded smooth domain. Let $r\geq 1$, $0<q\leq p\leq \infty$, $s>0$ be such that
\begin{equation}\label{eq:GNcondition}
 \frac1r\leq \frac1n+\frac1p.
\end{equation}
 Then there exists $c>0$ such that
 \[
  \norm[L^p(\Om)]{u}\leq c\left(\norm[L^r(\Om)]{\na u}^a \norm[L^q(\Om)]{u}^{1-a} + \norm[L^s(\Om)]{u}\right) \quad \mbox{ for all } u\in W^{1,r}(\Om)\cap L^q(\Om),
 \]
 where \begin{equation}\label{eq:GNdefa}
        a=\frac{\frac1q-\frac1p}{\frac1q+\frac1n-\frac1r}.
       \end{equation}
\end{lem}
\begin{proof}
 Compared to the standard version of the Gagliardo-Nirenberg inequality (\cite[p.126]{Nirenberg_ellipticPDE}, \cite[Thm. 10.1]{FriedmanPDE}), we include the possibility of $p,q\in(0,1)$, but restrict ourselves to the case of $q\leq p$. The conditions $q\leq p$ and $\frac1r\leq \frac1n+\frac1p$ ensure that $a\in[0,1]$. 
 For a proof of the case $1\leq p\leq q$, we refer to the aforementioned theorems of \cite{Nirenberg_ellipticPDE,FriedmanPDE}.
 For $q\leq1\leq p$, let $c:=\frac{\frac1q-1}{\frac1q-\frac1p}$. Then $c\in[0,1]$ 
 and $\frac cp+\frac{1-c}q=1$. H\"older's inequality with exponents $\frac cp$ and $\frac{1-c}q$ yields
$\norm[L^1(\Om)]{u}\leq \norm[L^p(\Om)]{u}^c\norm[L^q(\Om)]{u}^{1-c}$ and we conclude from the usual version of the Gagliardo-Nirenberg inequality that
\[
 \norm[L^p(\Om)]{u}\leq C_1 \norm[W^{1,r}(\Om)]{u}^b\norm[L^1(\Om)]{u}^{1-b} \leq C_1 \norm[W^{1,r}(\Om)]{u}^b \norm[L^p(\Om)]{u}^{c(1-b)} \norm[L^q(\Om)]{u}^{(1-c)(1-b)}
\]
for
\(
 b=\frac{1-\frac1p}{1+\frac1n-\frac1r},
\) and with some $C_1>0$
so that
\[
 \norm[L^p(\Om)]{u}\leq C_2
 \norm[W^{1,r}(\Om)]{u}^{\frac{b}{1-c(1-b)}}\norm[L^q(\Om)]{u}^{\frac{(1-c)(1-b)}{1-c(1-b)}}= C_2 \norm[W^{1,r}(\Om)]{u}^a\norm[L^q(\Om)]{u}^{1-a},
\]
where $C_2=C_1^{\frac1{1-c(1-b)}}$ and
\[
 a=\frac{b}{1-c+cb}=\frac{1-\frac1p}{1+\frac1n-\frac1r}\cdot\frac{1}{\frac{1-\frac1p}{\frac1q-\frac1p}+\frac{\frac1q-1}{\frac1q-\frac1p}\cdot\frac{1-\frac1p}{1+\frac1n-\frac1r}}=\frac{\frac1q-\frac1p}{1+\frac1n-\frac1r+\frac1q-1}
\]
coincides with the expression for $a$ given in \eqref{eq:GNdefa}.
For $q\leq p\leq 1$ let $c:=\frac{\frac1q-\frac1p}{\frac1q-1} \in [0,1]$. Then $cp+\frac{(1-c)p}q=1$ and by H\"older's inequality with exponents $cp$ and $\frac{(1-c)p}q$,
we have
\[\norm[L^p(\Om)]{u}\leq \norm[L^1(\Om)]{u}^c\norm[L^q(\Om)]{u}^{1-c}\leq C_3 \norm[W^{1,r}(\Om)]{u}^{cb} \norm[L^q(\Om)]{u}^{c(1-b)+1-c} = C_3 \norm[W^{1,r}(\Om)]{u}^a\norm[L^q(\Om)]{u}^{1-a}
\]
 with $C_3$ from the usual Gagliardo-Nirenberg inequality, where $b=\frac{\frac1q-1}{\frac1q+\frac1n-\frac1r}$ and thus $a:=cb=\frac{\frac1q-\frac1p}{\frac1q+\frac1n-\frac1r}$.
 In both cases, in $\norm[L^p(\Om)]{u}\leq \max\set{C_2,C_3} \norm[W^{1,r}(\Om)]{u}^a \norm[L^q(\Om)]{u}^{1-a}$ we may employ the Poincar\'e inequality (Lemma \ref{lem:poincare}) to obtain $C_P>0$, so that with $C_4=\max\set{C_2,C_3}\cdot C_P$ we achieve $\norm[L^p(\Om)]{u}\leq C_4 \norm[L^r(\Om)]{\na u}^a \norm[L^q(\Om)]{u}^{1-a} + C_4\norm[L^q(\Om)]{u}^{1-a}\norm[L^s(\Om)]{u}^a$, where $\norm[L^q(\Om)]{u}^{1-a}\norm[L^s(\Om)]{u}^a \leq C_5 \norm[L^s(\Om)]{u}$ if $s\geq q$ with some $C_5$ obtained from H\"older's inequality. If, on the other hand $s\leq q$, then $\norm[L^q(\Om)]{u}^{1-a} \norm[L^s(\Om)]{u}^a \leq C_6 \norm[L^q(\Om)]{u} \leq C_6 \norm[L^p(\Om)]{u}^d\norm[L^s(\Om)]{u}^{1-d} \leq \frac12 \norm[L^p(\Om)]{u} + C_7 \norm[L^s(\Om)]{u}$ with $C_6>0$ and $C_7>0$ from  H\"older's and Young's inequality, respectively, and $d=(\frac1s-\frac1q)/(\frac1s-\frac1p)$.
\end{proof}

When we have to estimate powers of norms, we will often without notice combine the Gagliardo-Nirenberg inequality with the following elementary estimate.
\begin{lem}
 For every $\alpha>0$ there is $C>0$ such that for all $x,y\in(0,\infty)$ we have
 \begin{equation*}
 (x+y)^\alpha\leq C(x^\alpha+y^\alpha).
 \end{equation*}
\end{lem}
\begin{proof}
 \(
  (x+y)^\alpha \leq (2\max\set{x,y})^\alpha \leq 2^\alpha \max\set{x^\alpha,y^\alpha} \leq 2^\alpha (x^\alpha+y^\alpha).
 \)
\end{proof}

Next let us give a basic property of the total mass of cancer cells that can be checked easily.

\begin{lem}\label{lem:massbd}
Let $T>0$, let $u\in C^{2,1}(\Ombar\times(0,T))\cap C^0(\Ombar\times[0,T))$ solve \eqref{ueq} with some functions $v,w\in C^{2,1}(\Ombar\times(0,T)\cap C^0(\Ombar\times[0,T))$, $D\in C^1(\R)$ and $\chi,\xi,\mu>0$. If $u$ and $w$ are nonnegative, $u$ satisfies
\begin{equation}\label{mass}
\io u(x, t)dx\leq m^{*}:=\max\set{|\Omega|, \io u(x,0)dx}
\end{equation}

for all $t\in(0, T)$.
\end{lem}
\begin{proof}
Thanks to the homogeneous boundary condition, we directly integrate \eqref{ueq} with respect to space.  Using the nonnegativity of $u$ and $w$ and the Cauchy-Schwarz inequality, we obtain
\begin{eqnarray*}
\frac{d}{dt}\io u=\mu\io -\mu\io u^2-\mu\io uw\leq\mu\io u-\mu\io u^2
\leq\mu\io u-\frac{\mu}{|\Omega|}\left(\io u\right)^2
\end{eqnarray*}
on $(0, T)$.  With an ODE comparison argument, this leads to (\ref{mass}).
\end{proof}

The following properties, turning information on the first solution component into boundedness assertions about the second, can be derived by invoking the variation-of-constants formula for $v$ and $L^p-L^q$ estimates for the heat semigroup.

\begin{lem}\label{lplqestimates}
Let $T\in(0,\infty]$, let $v\in C^{2,1}(\Ombar\times(0,T))\cap C^0(\Ombar\times[0,T))$ solve \eqref{veq} with some function $0\leq u\in C^0(\Ombar\times[0,T))$ and let $M>0$. \\
(i) Assume that $\norm[L^1(\Om)]{u(\cdot,t)}\leq M$ for all $t\in (0,T)$. Then for any $s\in [1, \frac{n}{n-2})$, there exists $c=c(M,s)>0$ such that
\begin{eqnarray*}
||v(\cdot, t)||_{L^s(\Omega)}\leq c \ for\ all\ t\in(0, T).
\end{eqnarray*}
(ii) Assume that $\norm[L^1(\Om)]{u(\cdot,t)}\leq M$ for all $t\in (0,T)$. Then for any $s\in[1, \frac{n}{n-1})$, there exists $c=c(M,s)>0$ such that
\begin{eqnarray*}
||\nabla v(\cdot, t)||_{L^{s}(\Omega)}\leq c \ for\ all\ t\in(0, T).
\end{eqnarray*}
(iii) Assume that $p>\frac{n}2$ and $\norm[L^p(\Om)]{u(\cdot,t)}\leq M$ for all $t\in (0,T)$. Then there is $c=c(p,M)>0$ such that
\[
 \norm[L^\infty(\Om)]{v(\cdot,t)}\leq c \mbox{ for all } t\in(0,T).
\]
(iv) Assume that $p>n$ and $\norm[L^p(\Om)]{u(\cdot,t)}\leq M$ for all $t\in (0,T)$. Then there is $c=c(p,M)>0$ such that
\[
 \norm[L^\infty(\Om)]{\nabla v(\cdot,t)}\leq c \mbox{ for all } t\in(0,T).
\]
\end{lem}

\begin{proof}
(i) Let $s\in[1,\frac n{n-2})$. We denote by $e^{t\Delta}$ the (Neumann-) heat semigroup and use its positivity to estimate
\begin{align*}
 \norm[L^s(\Om)]{v(\cdot,t)}\leq& \norm[L^s(\Om)]{e^{t(\Delta-1)} v_0} + \intnt \norm[L^s(\Om)]{e^{(t-\tau)(\Delta-1)}(u(\cdot,\tau)-\ubar(\cdot,\tau)+\ubar(\cdot,\tau))}d\tau\\
 \leq& \norm[L^s(\Om)]{\norm[L^\infty(\Om)]{v_0}} +\intnt e^{-(t-\tau)} \norm[L^s(\Om)]{\ubar(\cdot,\tau)} d\tau\\
 +&c_1\intnt(1+(t-\tau)^{-\frac n2(1-\frac1s)}) \norm[L^1(\Om)]{u(\cdot,\tau)-\ubar(\cdot,\tau)} e^{-\lambda_1(t-\tau)}d\tau\\
 \leq&|\Om|^{\frac1s} \norm[L^\infty(\Om)]{v_0} +\intnt \left(c_1(1+(t-\tau)^{-\frac n2(1-\frac1s)})e^{-\lambda_1(t-\tau)} 2M + e^{-(t-\tau)} M|\Om|^{\frac1s-1}\right) d\tau
\end{align*}
for all $t\in(0,T)$, where $c_1$ is the constant from the $L^p$-$L^q$-estimate in \cite[Lemma 1.3 (i)]{Winkler-10}, $\lambda_1>0$ is the first positive eigenvalue of $-\Delta$ in $\Om$, and $\ubar(\tau)=\frac1{|\Om|}\io u(x,\tau) dx$, $\tau\in(0,T)$, denotes the spatial average of $u$. Thanks to $s\in[1,\frac{n}{n-2})$, we have $-\frac n2(1-\frac1s)>-1$ and hence the last expression is bounded independently of $t\in(0,T)$.\\
(ii) Let $s\in[1, \frac{n}{n-1})$. From \cite[Lemma 1.3 (iii) and Lemma 1.3 (ii)]{Winkler-10} we obtain $c_2>0$ and $c_3>0$, respectively, such that
\begin{eqnarray*}
\|\nabla v(\cdot, t)\|_{L^{s}(\Omega)}
&\leq&c_2 \|\nabla e^{t(\Delta-1)}v_0\|_{L^{s}(\Omega)}+\int_0^t\|\nabla e^{(t-\tau)(\Delta-1)}u(\cdot, \tau)\|_{L^{s}(\Omega)}d\tau\nonumber\\
&\leq& c_2\|\nabla v_0\|_{L^{\infty}(\Omega)}+c_3\int_0^t(1+(t-\tau)^{-\frac{1}{2}-\frac{n}{2}(1-\frac{1}{s})})e^{-\lambda_1(t-\tau)}\|u(\cdot, \tau)\|_{L^1(\Omega)}d\tau
\end{eqnarray*}
for any $t\in(0,T)$. Again, the inequality $-\frac12-\frac n2(1-\frac1s)>-1$ and the boundedness assumption on $u$ provide us with a time-independent bound for this expression.\\
(iii) Let $p>\frac n2$. Lemma \cite[Lemma 1.3 (i)]{Winkler-10} gives $c_1>0$ such that
\begin{align*}
 \norm[L^\infty(\Om)]{v(t)}\leq& \norm[L^\infty(\Om)]{e^{t(\Delta-1)} v_0} + \intnt \norm[L^\infty(\Om)]{e^{(t-\tau)(\Delta-1)} u(\cdot,\tau)}d\tau\\
  \leq\norm[L^\infty(\Om)]{v_0} +\intnt &\left(c_1(1+(t-\tau)^{-\frac n{2p}})\norm[L^p(\Om)]{u(\cdot,\tau)-\ubar(\cdot,\tau)}e^{-\lambda_1(t-\tau)} +e^{-(t-\tau)}\norm[L^\infty(\Om)]{\ubar(\tau)} \right)d\tau\\
  \leq\norm[L^\infty(\Om)]{v_0} +\intnt &\left(c_1(1+(t-\tau)^{-\frac n{2p}})2Me^{-\lambda_1(t-\tau)} +e^{-(t-\tau)}|\Om|^{-1} \norm[L^1(\Om)]{u(\cdot,\tau)} \right) d\tau
\end{align*}
for every $t\in(0,T)$, which again results in boundedness due to $-\frac{n}{2p}>-1$.\\
(iv) Let $p>n$. Passing to the limit $p\to \infty$ in \cite[Lemma 1.3 (iii)]{Winkler-10} and from a variant of \cite[Lemma 1.3 (iv)]{Winkler-10} (given in the form needed here e.g. in \cite[Lemma 3.1]{Lankeit_exceed}), similarly to the previous part of the proof, we obtain $c_2>0$, $c_4>0$ such that
\begin{align*}
 \norm[L^\infty(\Om)]{\na v} \leq& \norm[L^\infty]{\na e^{t(\Delta-1)} v_0} +\intnt \norm[L^\infty]{\na e^{(t-\tau)(\Delta-1)} u(\cdot,\tau)} d\tau\\
  \leq& c_2\norm[L^\infty(\Om)]{\na v_0} +c_4 \intnt (1+(t-\tau)^{-\frac12-\frac n{2p}})e^{-\lambda_1(t-\tau)} \norm[L^p(\Om)]{u(\cdot,\tau)} d\tau
\end{align*}
for any $t\in(0,T)$, which again, because $-\frac12-\frac{n}{2p}>-1$, entails uniform-in-time boundedness of $\norm[L^\infty(\Om)]{\na v}$.
\end{proof}

Due to the fact that \eqref{weq} is an ODE, $w$ can be represented explicitly in terms of $v$. Following an observation from \cite{tao2014boundedness}, this provides a one-sided pointwise estimate for $-\Delta w$:

\begin{lem}\label{lem:bdDeltaw}
Assume that $\chi>0$, $\xi>0$ and $\mu>0$, $\alpha>0$, $T>0$, and let the nonnegative functions $v,w\in C^{2,1}(\Ombar\times(0,T))\cap C^0(\Ombar\times[0,T))$ solve \eqref{veq} and \eqref{weq}, respectively, with some nonnegative $u\in C^0(\Ombar\times[0,T))$ and let $w(\cdot,0)\in C^{2+\alpha}(\Om)$ be positive in $\Om$. Then
\begin{eqnarray*}
-\Delta w(x, t)\leq\|w(\cdot,0)\|_{L^{\infty}(\Omega)}\cdot v(x, t)+K\ \ for\ all\ x\in\Omega\ and\ t\in(0, T),
\end{eqnarray*}
where
\begin{equation}\label{eq:defK}
K:=\|\Delta w(\cdot, 0)\|_{L^{\infty}(\Omega)}+4\|\nabla\sqrt{w(\cdot,0)}\|_{L^{\infty}(\Omega)}^2+\frac{\|w(\cdot,0)\|_{L^{\infty}(\Omega)}}{e}.
\end{equation}
\end{lem}
\begin{proof}
This can be found in \cite[Lemma 2.2]{tao2014boundedness}.
\end{proof}

Let us finally recall the following local existence result, which has been established in \cite{tao2011chemotaxis} by means of a standard fixed point argument.

\begin{lem}\label{local existence}
(Local existence) Let $n\in\{2, 3, 4\}$ and $\Omega\subset \mathbb{R}^n$ be a bounded domain with smooth boundary.  Let $\chi, \xi, \mu>0$ and assume that $D$ satisfies (\ref{D12})-(\ref{D:nondeg}), and the initial data fulfils (\ref{eq:initdatacond}).  Then there exists a maximal time $T_{\mathrm{max}}\in(0, \infty]$ and a triple $(u, v, w)$ of functions from $C^0(\bar{\Omega}\times[0, T_{\mathrm{max}})\cap C^{2, 1}(\Ombar\times(0, T_{max})))$ solving \systemref\ classically in $\Omega\times(0, T_{\mathrm{max}})$ and such that
\begin{eqnarray}\label{T}
\mbox{either }\;\Tmax=\infty\;\, \mbox{ or }\;\, \limsup _{t\nearrow \Tmax}\left(\|u(\cdot, t)\|_{L^{\infty}(\Omega)}+\|v(\cdot, t)\|_{W^{1, \infty}(\Omega)}+\|w(\cdot, t)\|_{W^{1, \infty}(\Omega)}\right)=\infty.
\end{eqnarray}
Moreover,
\begin{eqnarray}\label{l1}
u\geq0,\ v\geq0,\ and\ 0\leq w\leq\|w_0\|_{L^{\infty}(\Omega)}\ in \ \Omega\times(0, T_{\mathrm{max}}).
\end{eqnarray}
\end{lem}

\begin{proof}
See \cite[Lemma 2.1]{tao2011chemotaxis}.
\end{proof}

\section{Boundedness of solutions. Proof of Theorem \ref{Theorem1}}\label{s3}
The purpose of this section is to prove Theorem \ref{Theorem1}, which is concerned with the non-degenerate case.

For the rest of this section let us fix $n\in\{2, 3, 4\}$ and the initial data satisfying (\ref{eq:initdatacond}). Given $\mu, \chi, \xi\ge 0$ and a function $D$, by $(u, v, w)$ we will always denote the corresponding solution to \systemref\ given by Lemma \ref{local existence} and by $T_{max}$ its maximal time of existence. The value of $K$ will be as defined by \eqref{eq:defK}.

Now we proceed to establish the main step towards our boundedness proof.  Motivated by \cite{Tao-Winkler-jde}, we establish a differential inequality for the expression $\io u^p+\io |\nabla v|^{2q}$ on $(0, T_{\mathrm{max}})$ for $p, q>1$, from which we will be able to derive a bound on $\io u^p+\io |\nabla v|^{2q}$.

\begin{lem}\label{l4} Let $\mu, \chi, \xi\ge 0$, $\delta>0$ and $m>1$. For any function $D$ satisfying \eqref{D12} and \eqref{D:nondeg}
and for any $p>1$ the solution $(u,v,w)$ to \systemref\ satisfies
\begin{align}\label{p1}
\frac{1}{p}\frac{d}{dt}&\io u^p+\frac{p-1}4\io D(u) u^{p-2} |\na u|^2 + \frac{\delta(p-1)}{(p+m-1)^2}\io |\na u^{\frac{p+m-1}2}|^2 \\\leq&\frac{\chi^2(p-1)}{\delta}\io u^{p-m+1}|\nabla v|^2
+\xi||w_0||_{L^{\infty}(\Omega)}\io u^p v+(\mu+\xi K)\io u^{p}-\mu\io u^{p+1}\nonumber
\end{align}
for all $t\in(0, T_{\mathrm{max}})$ and with $K$ as in \eqref{eq:defK}.
\end{lem}
\begin{proof}
Multiplying \eqref{ueq} by $u^{p-1}$ and integrating over $\Omega$ we obtain
\begin{eqnarray}\label{p11}
\frac{1}{p}\frac{d}{dt}\io u^p+(p-1)\io D(u) u^{p-2}|\nabla u|^2&\leq&\chi(p-1)\io u^{p-1}\nabla u\cdot\nabla v
+\xi(p-1)\io u^{p-1}\nabla u\cdot\nabla w\nn\\
&&+\mu\io u^p-\mu\io u^{p+1}-\mu\io u^pw
\end{eqnarray}
on $(0,\Tmax)$.
Here we split the integral containing $D$ into one part we will keep until we need it in Section \ref{s4}
and one part that can be used to cancel less favorable contributions by other integrals: According to \eqref{D12},
\begin{equation}\label{eq:splitD}
(p-1)\io D(u) u^{p-2}|\nabla u|^2 \geq \frac{p-1}2 \io D(u) u^{p-2}|\na u|^2 + \frac{\delta(p-1)}2 \io u^{m+p-3} |\na u|^2,
\end{equation}
which again holds on the whole time-interval $(0,\Tmax)$.
In the next integral in \eqref{p11} an application of Young's inequality yields
\begin{eqnarray}\label{交叉项1}
\chi(p-1)\io u^{p-1}\nabla u\cdot \nabla v\leq\frac{\delta(p-1)}{4}\io u^{m+p-3}|\nabla u|^2+\frac{\chi^2(p-1)}{\delta}\io u^{p-m+1}|\nabla v|^2
\end{eqnarray}
on $(0,\Tmax)$.
According to an integration by parts and Lemma \ref{lem:bdDeltaw}, we also have
\begin{eqnarray}\label{交叉项2}
\xi(p-1)\io u^{p-1}\nabla u\cdot\nabla w&=&\frac{\xi(p-1)}{p}\io u^p(-\Delta w)\nonumber \\
&&\leq\frac{\xi(p-1)}{p}\io u^p(||w_0||_{L^{\infty}(\Omega)}\cdot v+K)\nonumber\\
&&\leq\xi||w_0||_{L^{\infty}(\Omega)}\io u^p v+\xi K\io u^p \quad \mbox{ on }(0,\Tmax)
\end{eqnarray}
with $K$ as in \eqref{eq:defK}.
Substituting (\ref{交叉项1}), \eqref{eq:splitD} and (\ref{交叉项2}) into (\ref{p11}), we obtain
\begin{align*}
\frac{1}{p}\frac{d}{dt}&\io u^p+\frac{p-1}2\io D(u)u^{p-2}|\na u|^2+\frac{\delta(p-1)}{4}\io u^{p+m-3}|\na u|^2\\
 \leq&\frac{\chi^2(p-1)}{\delta}\io u^{p-m+1}|\nabla v|^2 +\xi||w_0||_{L^{\infty}(\Omega)}\io u^p v\nn\\
&+ \xi K\io u^{p} + \mu \int u^p -\mu\io u^{p+1} -\mu \io u^pw \qquad \mbox{ on }(0,\Tmax). \nonumber
\end{align*}
Finally inserting $u^{p+m-3}|\na u|^2 = (\frac{2}{p+m-1})^2 |\na u^{\frac{p+m-1}2}|^2$ and using the nonnegativity of $\io u^pw$, we arrive at \eqref{p1}.
\end{proof}

One of the terms on the right-hand side of \eqref{p1} that has to be dealt with is the integral $\io u^pv$. Since at the moment high powers of $u$ and $v$ are out of reach for our estimates, we will use the Gagliardo-Nirenberg inequality and gradient terms to control this integral:

\begin{lem}\label{l5}
Let $\mu, \chi, \xi\ge 0$. Let $\delta>0$ and $m>2-\frac4n$.
For all $p\in[1,\infty)$ and $\eta>0$ there exists a constant $C=C(p,\eta)>0$ such that for any function $D$ fulfilling \eqref{D12} and \eqref{D:nondeg} the solution to \systemref\ satisfies
\begin{eqnarray}\label{u^pv}
\xi||w_0||_{L^{\infty}(\Omega)}\io u^p(\cdot, t)v(\cdot, t)\leq \eta \norm[L^2(\Om)]{\nabla u^{\frac{p+m-1}{2}}(\cdot, t)}^2+ C
\end{eqnarray}
for all $t\in(0, T_{\mathrm{max}})$.
\end{lem}
\begin{proof}
Because $(n-4)p\leq0 < 2(m-1)$ and $m>2-\frac{4}{n}$ guarantee that
\begin{eqnarray*}
\frac{p+m-1}{2p\frac{n}{2}}>\frac{n-2}{2n}\quad \mbox{ and }\quad \frac{p-\frac2n}{\frac{p+m-1}2-\frac{n-2}{2n}} <2,
\end{eqnarray*}
it is possible to fix $\gamma\in(0, 1)$ so small that still
\begin{equation}\label{eq:choicer}
\frac{p+m-1}{2p\frac{n}{2-\gamma}}>\frac{n-2}{2n} \quad\mbox{and}\quad \frac{p-\frac{2-\gamma} n}{\frac{p+m-1}2-\frac{n-2}{2n}}<2.
\end{equation}
We set $r=\frac{n}{2-\gamma}$ and $r^{'}=\frac{n}{n-2+\gamma}$ and note that by Lemma \ref{lplqestimates}(i) applied to $s=r^{'}$, there exists $c>0$ such that $\norm[L^{r^{'}}(\Om)]{v(\cdot,t)}\leq c$ for all $t\in(0, \Tmax)$.  An application of H\"{o}lder's inequality then asserts
\begin{equation*}
\io u^p(\cdot,t)v(\cdot,t)\leq\bigg(\io u^{pr}(\cdot,t)\bigg)^{\frac{1}{r}}\!\!\cdot\bigg(\io v^{r^{'}}(\cdot,t)\bigg)^{\frac{1}{r^{'}}}\; \mbox{ for all } t\in (0,\Tmax).
\end{equation*}
As $0\leq \frac{2}{p+m-1}\leq \frac{2pr}{p+m-1}$ and, by \eqref{eq:choicer}, $\frac12\leq \frac1n+\frac{p+m-1}{2pr}$, from the Gagliardo-Nirenberg inequality and Lemma \ref{lem:massbd}, we obtain constants $c_1, c_2>0$ such that
\begin{eqnarray*}
\bigg(\io u^{pr}\bigg)^{\frac{1}{r}}&=&||u^{\frac{p+m-1}{2}}||_{L^{\frac{2pr}{p+m-1}}(\Omega)}^{\frac{2p}{p+m-1}}\nonumber\\
&\leq& c_1\left (||\nabla u^{\frac{p+m-1}{2}}||_{L^2(\Omega)}^{\frac{2p}{p+m-1}\cdot a}||u^{\frac{p+m-1}{2}}||_{L^{\frac{2}{p+m-1}}}^{\frac{2p}{p+m-1}(1-a)}+||u^{\frac{p+m-1}{2}}||_{L^{\frac{2}{p+m-1}}}^{\frac{2p}{p+m-1}}\right)\nonumber\\
&\leq& c_2\left(\norm[L^2(\Om)]{\nabla u^{\frac{p+m-1}{2}}}^{\frac{2p}{p+m-1}\cdot a}+1\right) \qquad \mbox{ on }(0,\Tmax)
\end{eqnarray*}
with
\begin{eqnarray*}
a=\frac{\frac{p+m-1}{2}-\frac{p+m-1}{2pr}}{\frac{p+m-1}{2}+\frac1n-\frac12}=\frac{p+m-1}2\cdot \frac{1-\frac{2-\gamma}{pn}}{\frac{p+m-1}{2}-\frac{n-2}{2n}}.
\end{eqnarray*}
Since
\begin{eqnarray*}
\frac{2p}{p+m-1}\cdot a= \frac{p-\frac{2-\gamma}n}{\frac{p+m-1}2-\frac{n-2}{2n}}<2,
\end{eqnarray*}
by \eqref{eq:choicer}, an application of Young's inequality thus completes the proof: Given $\eta>0$ it provides $C>0$ such that
\begin{align*}
\xi\norm[L^\infty(\Om)]{w_0}\io u^p(\cdot,t)v(\cdot,t)\leq&\xi \norm[L^\infty(\Om)]{w_0} c c_2 \left(\norm[L^2(\Om)]{\nabla u^{\frac{p+m-1}{2}}(\cdot,t)}^{\frac{2p}{p+m-1}\cdot a}+1\right)\\
\leq& \eta \norm[L^2(\Om)]{\nabla u^{\frac{p+m-1}{2}}(\cdot, t)}^2+ C
\end{align*}
for all $t\in(0,\Tmax)$.
\end{proof}

The next part of the differential inequality we are searching for is given by the following lemma.

\begin{lem}\label{l6}
Let $\mu, \chi, \xi\ge 0$. Let $\delta>0$ and $m>1$. For $q\in[1,\infty)$ there exists a constant $C=C(q)>0$ such that for any function $D$ with \eqref{D12} and \eqref{D:nondeg}, the solution to \systemref\ obeys
\begin{eqnarray}\label{nabla v}
\frac{1}{q}\frac{d}{dt}\io |\nabla v|^{2q}+2\io|\nabla v|^{2q}+\frac{(q-1)}{q^2}\io |\nabla|\nabla v|^q|^2\leq C\io u^2|\nabla v|^{2q-2}+C
\end{eqnarray}
on $(0, \Tmax)$.
\end{lem}
\begin{proof}
This can be proven by merging the estimate from \cite[Prop. 3.2]{ishida2014boundedness} with some of the calculations from the proof of \cite[Lemma 3.3]{Tao-Winkler-jde}:
We first observe that
\begin{equation}\label{eq:ddtionavq}
 \frac{1}{q}\frac{d}{dt}\io |\nabla v|^{2q} = 2\io |\na v|^{2q-2}\na v\na\Delta v-2\io |\na v|^{2q-2}\na v\cdot\na v+2\io |\na v|^{2q-2}\na v\cdot\na u
\end{equation}
holds throughout $(0,\Tmax)$, and use the pointwise identities
\begin{equation} \label{eq:pwid}
 2\na v\na\Delta v = \Delta |\na v|^2 - 2|D^2 v|^2\qquad \mbox{ and } \qquad \na |\na v|^{2q-2}=(q-1)|\na v|^{2q-4}\na|\na v|^2
\end{equation}
together with an integration by parts to obtain
\begin{align*}
 2\io |\na v|^{2q-2}\na v\na\Delta v=&\io |\na v|^{2q-2}\Delta |\na v|^2 - 2\io |\na v|^{2q-2}|D^2 v|^2\\
=& - (q-1) \io |\na v|^{2q-4} | \na |\na v|^2 |^2 + \intdom |\na v|^{2q-2}\frac{\partial |\na v|^2}{\partial\nu}\\
&- 2\io |\na v|^{2q-2}|D^2 v|^2\qquad \mbox{ on } (0,\Tmax).
\end{align*}
Another integration by parts, \eqref{eq:pwid} and Young's inequality make it possible to estimate the rightmost term in \eqref{eq:ddtionavq} according to
\begin{align*}
 2\io |\na v|^{2q-2} \na v\cdot \na u =& -2(q-1) \io u|\na v|^{2q-4}\na|\na v|^2 \cdot\na v - 2\io u |\na v|^{2q-2}\Delta v \\
 \leq & \frac{q-1}2 \io |\na v|^{2q-4}|\na|\na v|^2|^2 + 2(q-1)\io u^2|\na v|^{2q-2}\\ &+ \frac2n \io |\na v|^{2q-2}n|D^2v|^2 + \frac n2\io u^2|\na v|^{2q-2} \qquad \mbox{ on } (0,\Tmax),
\end{align*}
where we have used $|\Delta v|^2\leq n|D^2 v|^2$ in the second-last integral.
Adding these estimates, we arrive at
\begin{align*}
 \frac{1}{q}\frac{d}{dt}\io |\nabla v|^{2q} +2\io |\na v|^{2q} \leq&
- \frac{q-1}2 \io |\na v|^{2q-4} | \na |\na v|^2 |^2  \\&+ (2(q-1)
+\frac n2) \io u^2|\na v|^{2q-2} + \intdom |\na v|^{2q-2}\frac{\partial |\na v|^2}{\partial\nu}
\end{align*}
on $(0,\Tmax)$ and thus at \eqref{nabla v} if we take into account that $|\na v|^{2q-4}|\na|\na v|^2|^2=\frac4{q^2}|\na |\na v|^q|^2$ and use Lemma \ref{lem:bdryestimate}, which is applicable by a combination of Lemma \ref{lplqestimates} (ii) and Lemma \ref{lem:massbd}, to gain $c>0$ such that
\[
 \intdom |\na v|^{2q-2} \frac{\partial |\na v|^2}{\partial \nu} \leq \frac{q-1}{q^2}\io |\na|\na v|^q|^2 + c.\qquad \mbox{on } (0,\Tmax)\qedhere
\]
\end{proof}

The terms on the right-hand side of \eqref{p1} we have not yet treated are $+(\xi K+\mu) \io u^p - \mu \io u^{p+1}$. For positive $\mu$ they could easily be estimated from above by a constant. In order to also cover the case of $\mu=0$, we prepare the following estimate instead:

\begin{lem}\label{lem:intup}
 Let $\mu,\chi,\xi\ge 0$. Let $\delta>0$ and $m>1-\frac2n$. Then for any $p>1$ and $\eta>0$ there is a constant $C=C(\eta)>0$ such that for any function $D$ with \eqref{D12} and \eqref{D:nondeg}, the solution to \systemref\ satisfies
\[
 \io u^p \leq \eta \io |\na u^{\frac{p+m-1}2}|^2+ C
\]
 on $(0,\Tmax)$.
\end{lem}
\begin{proof}
Once more, an application of Lemma \ref{lem:massbd} shows that $\io u(\cdot,t) \leq c_1$ for some $c_1>0$. 
Obviously, $\frac{2p}{p+m-1}>\frac{2}{p+m-1}$ and $\frac12<\frac1n+\frac12+\frac{m-1}{2p} = \frac1n+\frac{p+m-1}{2p}$, 
therefore the Gagliardo-Nirenberg inequality asserts the existence of $c_2>0$ such that 
\begin{align*}
  \io u^p =& \norm[L^{\frac{2p}{p+m-1}}(\Om)]{u^{\frac{p+m-1}2}}^{\frac{2p}{p+m-1}}\\
 \leq& c_2 \left(\norm[L^2(\Om)]{\na u^{\frac{p+m-1}2}}^{\frac{2p}{p+m-1}a}\norm[L^{\frac2{p+m-1}}(\Om)]{u^{\frac{p+m-1}2}}^{\frac{2p}{p+m-1}(1-a)}
 +\norm[L^{\frac{2}{p+m-1}}(\Om)]{u^{\frac{p+m-1}{2}}}^{\frac{2p}{p+m-1}} \right)\\
 \leq& c_2c_1^{\frac{2p}{p+m-1}(1-a)}\left(\io |\na u^{\frac{p+m-1}2}|^2\right)^{\frac{p}{p+m-1}a} +c_2c_1^{\frac{2p}{p+m-1}}
\end{align*}
on $(0,\Tmax)$, where 
\[
 a=\frac{\frac{p+m-1}{2}-\frac{p+m-1}{2p}}{\frac{p+m-1}{2}+\frac1n-\frac12}
\]
and 
\[
 \frac{p}{p+m-1}a=\frac{p}{p+m-1} \cdot \frac{(p+m-1)(1-\frac1p)}{p+m-1+\frac2n-1}= \frac{p-1}{p-1+(m-1+\frac2n)}<1
\]
so that an application of Young's inequality gives the desired conclusion.
\end{proof}

What we have achieved with the previous estimates is the following:

\begin{cor}\label{c7}
Let $\mu, \chi, \xi\ge 0$. Let $\delta>0$ and $m>2-\frac4n$. For any $1<p, q<\infty$ there exists a constant $C=C(p, q)>0$ such that for any $D$ fulfilling \eqref{D12} and \eqref{D:nondeg}, the solution to \systemref\ satisfies
\begin{align}\label{eq:diffineq_prelim}
\frac{d}{dt}&\bigg\{\io u^p+\io |\nabla v|^{2q}\bigg\}+\frac{p(p-1)}4\io D(u)u^{p-2}|\na u|^2 + \frac{\delta p(p-1)}{2(p+m-1)^2}\io |\nabla u^{\frac{p+m-1}{2}}|^2\nn\\
&+2q\io|\nabla v|^{2q}+\frac{(q-1)}{q}\io \left|\nabla|\nabla v|^q\right|^2
\leq C\io u^{p-m+1}|\nabla v|^2+C\io u^2|\nabla v|^{2q-2}+C
\end{align}
on the whole time-interval $(0, T_{\mathrm{max}})$.
\end{cor}
\begin{proof}
We employ Lemma \ref{l5} and Lemma \ref{lem:intup} with $\eta=\frac{\delta (p-1)}{4(p+m-1)^2}$ or $\eta=\frac{\delta(p-1)}{4(p+m-1)^2(\xi K+\mu)}$, respectively,  as well as Lemma \ref{l4} and Lemma \ref{l6} and add the resulting inequalities. 
\end{proof}

In a similar manner as in \cite{Tao-Winkler-jde}, we deal with the two terms on the right hand side of \eqref{eq:diffineq_prelim}. We will have a closer look at the conditions on $p,q$ in Lemma \ref{lem:parameterchoice} afterwards.

\begin{lem}\label{l8}
Let $\mu, \chi, \xi\ge 0$, $\delta>0$ and $m>2-\frac2n$.
Let $p,q>1$ be such that
$p> m -\frac{n-2}{nq}$ and
\begin{eqnarray}\label{eq:condonp_plarge}
 1 > \frac{q}{q-1}\frac{p-m+1}{p+m-1}\left((p+m-1)\frac{n-\frac{nq-n+2}{q(p-m+1)}}{(p+m-1)n+2-n}\right)
\end{eqnarray}

Then for any $\eta>0$, one can find a constant $C=C(\eta, p, q)>0$ such that for any $D$ fulfilling \eqref{D12} and \eqref{D:nondeg}, the solution to \systemref\ satisfies
\begin{eqnarray}\label{E1}
\io u^{p-m+1}|\nabla v|^2\leq \eta\io |\nabla u^{\frac{p+m-1}{2}}|^2+\eta\io |\nabla|\nabla v|^q|^2+C
\end{eqnarray}
for all $t\in(0, T_{\mathrm{max}})$.
\end{lem}
\begin{proof}
Because $p\geq m-\frac{n-2}{nq}$, we have $nqp\geq mnq-n+2$ and thus $nq-n+2\leq nq+npq-nqm=nq(p-m+1)$, which results in $0\leq \frac2{p+m-1}\leq \frac{2nq(p-m+1)}{(p+m-1)(nq-n+2)}$. Furthermore, $q\geq 1$ so that $n-2-2q\leq 0$ and hence $-2pq+p(n-2)\leq 0\leq nq(m-1)+(q-1)(n-2)(m-1)$, which results in
$q(n-2)(p-m+1)=pqn-2pq-q(n-2)(m-1) \leq pqn-p(n-2)+nq(m-1)-(n-2)(m-1)=(nq-n+2)(p+m-1)$ and thus $\frac12-\frac1n=\frac{n-2}{2n}\leq \frac{(nq-n+2)(p+m-1)}{2nq(p-m+1)}$.

All in all, by these estimates and \eqref{eq:condonp_plarge}, it is possible to choose $\gamma>0$ in such a way that
\begin{equation}\label{eq:choosegamma_1}
 \frac{p-m+1}{p+m-1}\left((p+m-1)\frac{n-\frac{nq-n+2-\gamma}{q(p-m+1)}}{(p+m-1)n+2-n}\right)<\frac{q-1}q.
\end{equation}
and
\begin{equation}\label{eq:choosegamma_2}
 \frac12-\frac1n\leq \frac{(nq-n+2-\gamma)(p+m-1)}{2nq(p-m+1)} \quad \mbox{and} \quad nq-n+2-\gamma\leq nq(p-m+1).
\end{equation}

Our main purpose in introducing $\gamma$ is to avoid $n-2$ in the denominator in the $2$-dimensional case. For $n\in\set{3,4}$ also using $\gamma=0$ would be possible. Invoking H\"{o}lder's inequality with the exponents $\frac{nq}{nq-n+2-\gamma}$ and $\frac{nq}{n-2+\gamma}$ we find that
\begin{eqnarray*}
\io u^{p-m+1}|\nabla v|^2\leq \left(\io u^{\frac{nq(p-m+1)}{nq-n+2-\gamma}}\right)^{\frac{nq-n+2-\gamma}{nq}}\left(\io |\nabla v|^{\frac{2nq}{n-2+\gamma}}\right)^{\frac{n-2+\gamma}{nq}}\qquad \mbox{ on }(0,\Tmax).
\end{eqnarray*}
From the Sobolev embedding $W^{1, 2}(\Omega)\hookrightarrow L^{\frac{2n}{n-2+\gamma}}(\Omega)$ and Lemma \ref{lem:poincare} we obtain $c_1>0$ and $c_2>0$, respectively, such that for arbitrary $s_0\in[1,\frac{n}{n-1})$ we have
\begin{align*}
\left(\io |\nabla v|^{\frac{2nq}{n-2+\gamma}}\right)^{\frac{n-2+\gamma}{nq}} =& \||\nabla v|^q\|_{L^{\frac{2n}{n-2+\gamma}}(\Omega)}^{\frac{2}{q}}
\leq c_1\||\nabla v|^q\|_{W^{1, 2}(\Omega)}^{\frac{2}{q}} \\
\leq& c_2\left(\|\nabla|\nabla v|^q\|_{L^{2}(\Omega)}^{\frac{2}{q}}+\||\nabla v|^q\|_{L^{\frac{s_0}{q}}(\Omega)}^{\frac{2}{q}}\right)
\leq c_2 \norm[L^2(\Om)]{\na |\na v|^q}^{\frac2q}+c_3
\end{align*}
on $(0,\Tmax)$, where $\sqrt{\frac{c_3}{c_2}}$ is the contant given by Lemma \ref{lplqestimates} (ii) corresponding to $s_0$.

Thanks to the first inequality in \eqref{eq:choosegamma_2}, the Gagliardo-Nirenberg inequality therupon provides us with $c_4>0$ such that
\begin{eqnarray}\label{eq:351}
\left(\io u^{\frac{nq(p-m+1)}{nq-n+2-\gamma}}\right)^{\frac{nq-n+2-\gamma}{nq}}
&=&||u^{\frac{p+m-1}{2}}||_{L^{\frac{2nq(p-m+1)}{(p+m-1)(nq-n+2-\gamma)}}(\Omega)}^{\frac{2(p-m+1)}{p+m-1}}\nonumber\\
&\leq&c_4 \bigg(||\nabla u^{\frac{p+m-1}{2}}||_{L^2(\Omega)}^{\frac{2a(p-m+1)}{p+m-1}}||u^{\frac{p+m-1}{2}}||_{L^{\frac{2}{p+m-1}}(\Omega)}^{\frac{2(1-a)(p-m+1)}{p+m-1}}\nonumber\\
&& \qquad\qquad +||u^{\frac{p+m-1}{2}}||_{L^{\frac{2}{p+m-1}}(\Omega)}^{\frac{2(p-m+1)}{p+m-1}}\bigg),
\end{eqnarray}
on $(0,\Tmax)$, where
\begin{eqnarray*}
a=\frac{\frac{p+m-1}{2}-\frac{(p+m-1)(nq-n+2-\gamma)}{2nq(p-m+1)}}{\frac{p+m-1}{2}+\frac1n-\frac12}.
\end{eqnarray*}

Thanks to the boundedness of $\io u$ by Lemma \ref{lem:massbd}, we may continue \eqref{eq:351} to estimate
\begin{equation*}
\left(\io u^{\frac{nq(p-m+1)}{nq-n+2-\gamma}}\right)^{\frac{nq-n+2-\gamma}{nq}}\leq c_4||\nabla u^{\frac{p+m-1}{2}}||_{L^2(\Omega)}^{\frac{2a(p-m+1)}{p+m-1}}+c_5\qquad \mbox{ on }(0,\Tmax)
\end{equation*}
with appropriate $c_5>0$ obtained from Lemma \ref{lem:massbd}.

Given $\eta>0$, Young's inequality gives $c_6>0$ such that, in conclusion, 
\begin{align}
 \io& u^{p-m+1} |\na v|^2 \leq c_2c_4 \left(\io |\na|\na v|^q|^2\right)^{\frac1q} \left(\io |\na u^{\frac{p+m-1}2}|^2\right)^{a\frac{p-m+1}{p+m-1}} \nn\\
 &+ c_3c_4 \left(\io |\na u^{\frac{p+m-1}2}|^2\right)^{a\frac{p-m+1}{p+m-1}}
+ c_2c_5 \left(\io |\na|\na v|^q|^2\right)^{\frac1q}
+c_3c_5\nn\\
\leq& c_2c_4 \left(\io |\na|\na v|^q|^2\right)^{\frac1q} \left(\io |\na u^{\frac{p+m-1}2}|^2\right)^{a\frac{p-m+1}{p+m-1}}\nn\\
&+\frac{\eta}2 \io |\na|\na v|^q|^2 +\frac\eta2 \io |\na u^{\frac{p+m-1}2}|^2
+c_6\label{eq:upm1nav2est}\qquad \mbox{ on }(0,\Tmax).
\end{align}
Furthermore, by \eqref{eq:choosegamma_1}, we also have
\begin{eqnarray*}
\frac{a(p-m+1)}{p+m-1}<\frac{q-1}{q}.
\end{eqnarray*}
Hence in the first term on the right hand side of \eqref{eq:upm1nav2est} we can employ Young's inequality to obtain $c_7>0$ such that 
\begin{align*}
c_2c_4 \left(\io |\na|\na v|^q|^2\right)^{\frac1q} \left(\io |\na u^{\frac{p+m-1}2}|^2\right)^{a\frac{p-m+1}{p+m-1}}
&\leq\frac{\eta}2\io |\na|\na v|^q|^2+\frac{\eta}2 \io |\na u^{\frac{p+m-1}2}|^2\nn+c_7
\end{align*}
to finally transform \eqref{eq:upm1nav2est} into \eqref{E1}.
\end{proof}

We can treat the second term on the right-hand side of \eqref{eq:diffineq_prelim} similarly as the first one in Lemma \ref{l8}:

\begin{lem}\label{l9}
Let $\mu, \chi, \xi\ge 0$, $\delta>0$ and $m>2-\frac2n$.
Let $q>1$ and $p>\max\{1, \frac{2q(n-2)}{2q+n-2}-m+1\}$ be such that
\begin{equation}\label{eq:pqcond}
 \frac1q > \frac2{p+m-1} \left(\frac{\frac{p+m-1}2 - \frac{(p+m-1)(2q+n-2)}{4nq}}{\frac{p+m-1}2+\frac1n-\frac12} \right).
\end{equation}
Then for any $\eta>0$ there exists a constant $C>0$ such that for any $D$ with \eqref{D12} and \eqref{D:nondeg} the solution to \systemref\ fulfils
\begin{eqnarray}\label{E2}
\io u^2|\nabla v|^{2q-2}\leq\eta\io |\nabla u^{\frac{p+m-1}{2}}|^2+\eta\io |\nabla|\nabla v|^q|^2+C
\end{eqnarray}
for all $t\in(0, T_{\mathrm{max}})$.
\end{lem}

\begin{proof}
Since $(n-2)(q-1)\geq 0$ implies $\frac{nq}{2q+n-2}\geq 1$, and due to \eqref{eq:pqcond}, we can find $\gamma>0$ such that $\frac{nq}{2q+n-2-\gamma(q-1)}>1$ and
\begin{equation}\label{eq:choosegamma2_1}
 \frac{2}{p+m-1}\left(\frac{\frac{p+m-1}2-\frac{(p+m-1)(2q+n-2-\gamma(q-1))}{4nq}}{\frac{p+m-1}2+\frac1n-\frac12}\right)< \frac1q.
\end{equation}
Using H\"{o}lder's inequality with exponents $\frac{nq}{2q+n-2-\gamma(q-1)}$ and $\frac{nq}{nq-2q-n+2+\gamma(q-1)}=\frac{nq}{(n-2+\gamma)(q-1)}$ yields
\begin{eqnarray*}
\io u^2|\nabla v|^{2q-2}\leq\left(\io u^{\frac{2nq}{2q+n-2-\gamma(q-1)}}\right)^{\frac{2q+n-2-\gamma(q-1)}{nq}}\left(\io |\nabla v|^{\frac{2nq}{n-2+\gamma}}\right)^{\frac{(q-1)(n-2+\gamma)}{nq}}\qquad \mbox{ on }(0,\Tmax),
\end{eqnarray*}
where we may again use Sobolev's and Poincar\'{e}'s inequalities to obtain
\begin{eqnarray*}
\left(\io |\nabla v|^{\frac{2nq}{n-2+\gamma}}\right)^{\frac{(q-1)(n-2+\gamma)}{nq}}=\||\nabla v|^q|\|_{L^{\frac{2n}{n-2+\gamma}}(\Omega)}^{\frac{2q-2}{q}}
&\leq& c_1\||\nabla v|^q\|_{W^{1, 2}(\Omega)}^{\frac{2q-2}{q}}\nonumber\\
&\leq& c_2\left(\|\nabla|\nabla v|^q\|_{L^2(\Omega)}^{\frac{2q-2}{q}}+\||\nabla v|^q\|_{L^{\frac{s_0}{q}}}^{\frac{2}{q}}\right)
\end{eqnarray*}
for all $t\in(0, \Tmax)$ with $s_0\in[1, \frac{n}{n-1})$.
We have $2nq\geq nq+nq\geq 2q+n\geq 2q+n-2$ so that $0\leq \frac2{p+m-1}\leq \frac{4nq}{(p+m-1)(2q+n-2)}$. Moreover, $\frac{2q(n-2)}{2q+n-2}-m+1 \leq p$ entails $\frac12=\frac1n+\frac{n-2}{2n} = \frac1n+\frac{\frac{2q(n-2)}{2q+n-2}(2q+n-2)}{4nq}\leq \frac1n+\frac{(p+m-1)(2q+n-2)}{4nq}$ and hence the Gagliardo-Nirenberg inequality yields $c_3>0$ with
\begin{eqnarray*}
\left(\io u^{\frac{2nq}{2q+n-2-\gamma(q-1)}}\right)^{\frac{2q+n-2-\gamma(q-1)}{nq}}
&=&\|u^{\frac{p+m-1}{2}}\|_{L^{\frac{4nq}{(p+m-1)(2q+n-2-\gamma(q-1))}}(\Omega)}^{\frac{4}{p+m-1}}\nonumber\\
&\leq&c_3\bigg(\|\nabla u^{\frac{p+m-1}{2}}\|_{L^2(\Omega)}^{\frac{4b}{p+m-1}}\|u^{\frac{p+m-1}{2}}\|_{L^{\frac{2}{p+m-1}}(\Omega)}^{\frac{4(1-b)}{p+m-1}}\nonumber\\
&&+\|u^{\frac{p+m-1}{2}}\|_{L^{\frac{2}{p+m-1}}(\Omega)}^{\frac{4}{p+m-1}}\bigg)\nonumber\\
&\leq& c_4\left(\|\nabla u^{\frac{p+m-1}{2}}\|_{L^2(\Omega)}^{\frac{4b}{p+m-1}}+1\right)\qquad \mbox{ on }(0,\Tmax),
\end{eqnarray*}
where we have used Lemma \ref{lem:massbd} to find $c_4>0$ and where
\begin{eqnarray*}
b=\frac{\frac{p+m-1}{2}-\frac{(p+m-1)(2q+n-2-\gamma(q-1))}{4nq}}{\frac{p+m-1}{2}+\frac1n-\frac12}.
\end{eqnarray*}
Here \eqref{eq:choosegamma2_1} ensures that
\(
\frac{2b}{p+m-1}<\frac{1}{q}
\)
and concluding as in the previous proof, by twofold application of Young's inequality (with exponents $\frac{p+m-1}{2b}$, $\frac{q}{q-1}$ in the first step) 
we obtain \eqref{E2}.
\end{proof}

The conditions in the previous lemmata involve some assumptions on $p$ and $q$ that are not immediately seen to be simultaneously satisfiable. With this lemma we ensure that they are. In its proof we rely on the fact that $m>2-\frac2n$.

\begin{lem}\label{lem:parameterchoice}
 Let $m>2-\frac2n$. 
 There exist unbounded sequences $(p_k)_{k\in\N}, (q_k)_{k\in\N}$ such that for each $k\in\N$
\begin{align}
 q_k>&1\qquad \mbox{ and }\qquad p_k>1 \label{pq:gtone}\\
 p_k>&\max\set{\frac{2q_k(n-2)}{2q_k+n-2}-m+1,m-\frac{n-2}{nq_k}}\label{pq:plarge_final}\\
 1 >& \frac{q_k}{q_k-1}\frac{p_k-m+1}{p_k+m-1}\left((p_k+m-1)\frac{n-\frac{nq_k-n+2}{q_k(p_k-m+1)}}{(p_k+m-1)n+2-n}\right) \label{pq:l8}\\
 \frac1q_k >& \frac2{p_k+m-1} \left(\frac{\frac{p_k+m-1}2 - \frac{(p_k+m-1)(2q_k+n-2)}{4nq_k}}{\frac{p_k+m-1}2+\frac1n-\frac12} \right) \label{pq:l9}.
\end{align}
and hence $p_k$ and $q_k$ fulfil all requirements of $p$ and $q$ from Lemmata \ref{l8} and \ref{l9}.
\end{lem}
\begin{proof}
In a first step we will show that it is possible to choose $p_k$, $q_k$ in such a way that the conditions
\begin{align}
p_k>&\max\set{\frac{2q_k(n-2)}{2q_k+n-2}-m+1, 2q_k\frac{n-1}n-m+1,m-\frac{n-2}{nq_k}}\label{pq:plarge}\\
 p_k<&(2q_k-1)(m-1)+\frac{2q_k}n\label{pq:psmall}
\end{align}
are satisfied, which evidently already imply \eqref{pq:plarge_final}, but, in a second step, can also be used to derive \eqref{pq:l8} and \eqref{pq:l9} as well. 

There is $Q_0>0$ such that for every $q>Q_0$
\[
 \max\set{\frac{2q(n-2)}{2q+n-2}-m+1, \frac{2q(n-1)}n-m+1,m-\frac{n-2}{nq}} = 2q\frac{n-1}n-m+1
\]
as can be seen by considering the limit of each of the expressions as $q\to\infty$.
By the condition on $m$ 
 we know that $m+\frac2n-2>0$ and hence it is possible to find $Q_1>0$ such that for every $q>Q_1$ we have
\[
 2q\left(m+\frac2n-2\right)>m-1
\]
and hence
\[
 2qm-m-2q+1+\frac{2q}n>-\frac{2q}n+2q, \mbox{i.e. }\quad (2q-1)(m-1)+\frac{2q}n>\frac{2q(n-1)}n,
\]
so that it becomes possible to choose unbounded sequences fulfilling \eqref{pq:gtone}, \eqref{pq:plarge} and \eqref{pq:psmall} simultaneously. This is already sufficient for them to satisfy \eqref{pq:l8} and \eqref{pq:l9} also. Indeed, \eqref{pq:psmall} ensures that
\[
 np_k<(2q_k-1)(m-1)n+2q_k=2q_k+2nmq_k-2nq_k-mn+n
\]
and hence
\begin{align*}
 &npq_k+nmq_k-2nq_k+2q_k-np_k-mn+2n-2\\>& np_kq_k+nmq_k-2nq_k+2q_k-(2q_k+2nmq_k-2nq_k-mn+n)-mn+2n-2\\
 >& np_kq_k-nmq_k+n-2,
\end{align*}
so that
\begin{align*}
 1>&\frac{np_kq_k-nmq_k+n-2}{np_kq_k+nmq_k-2nq_k+2q_k-np_k-mn+2n-2} = \frac{nq_k(p_k-m+1)-(nq_k-n+2)}{(q_k-1)(np_k+nm-2n+2)}\\
 =&\frac{q_k}{q_k-1}\frac{p_k-m+1}{p_k+m-1}\left((p_k+m-1)\frac{n-\frac{nq_k-n+2}{q_k(p_k-m+1)}}{(p_k+m-1)n+2-n}\right).
\end{align*}
Moreover, \eqref{pq:plarge} entails
\[
 2q_k\frac{n-1}n+1-m<p_k,\quad \mbox{ i.e. } \quad  2q_k(n-1)+n-mn<np_k
\]
and thus
\[
 2nq_k-2q_k-n+2<np_k+mn-2n+2,
\]
which shows that
\begin{align*}
 1>&\frac{2nq_k-2q_k-n+2}{np_k+mn-n+2-n}=\frac{q_k-\frac{2q_k+n-2}{2n}}{\frac{p_k+m-1}2+\frac1n-\frac12}
 =\frac{2q_k}{p_k+m-1}\left(\frac{\frac{p_k+m-1}2-\frac{(p_k+m-1)(2q_k+n-2)}{4nq_k}}{\frac{p_k+m-1}2+\frac1n-\frac12}\right).
\end{align*}
\end{proof}

Having made sure that there are (arbitrarily large) exponents, for which the lemmata before can be applied, we now use them in order to obtain the following boundedness result.

\begin{lem}\label{l10}
Let $\mu, \chi, \xi\ge 0$, $\delta>0$ and $m>2-\frac2n$. Let $p, q\in(1, \infty)$. Then one can find a constant $C>0$ depending on $p, q$ such that for any function $D$ satisfying \eqref{D12} and \eqref{D:nondeg}, we have
\begin{eqnarray}\label{345}
\io u^p(\cdot, t)+\io |\nabla v|^{2q}(\cdot, t)\leq C
\end{eqnarray}
for all $t\in(0, T_{\mathrm{max}})$.
\end{lem}
\begin{proof}
Observing $L^r(\Omega)\hookrightarrow L^s(\Omega)$ for $r>s$, we may without loss of generality enlarge $p$ and $q$ and do so in a manner that all conditions on $p$ and $q$ listed in Lemma \ref{l8} and \ref{l9} are satisfied, which is possible due to Lemma \ref{lem:parameterchoice}. 
An application of Corollary \ref{c7} provides us with constants $c_1, c_2, c_3, c_4>0$ such that 
\begin{align*}
 \ddt &\left(\io u^p + \io |\na v|^{2q}\right)+c_1 \io |\na u^{\frac{p+m-1}{2}}|^2 +c_2 \io |\na v|^{2q} + c_3\io |\na |\na v|^q|^2 \\
 &\leq c_4\io u^{p-m+1}|\na v|^2 + c_4\io u^2|\na v|^{2q-2} + c_4
\end{align*}
holds on $(0,\Tmax)$. Invoking Lemma \ref{l8} and Lemma \ref{l9} upon the choice of $\eta=\min\set{\frac{c_1}{4c_4},\frac{c_3}{2c_4}}$ we find $c_5>0$ such that 
\begin{equation}\label{eq:328}
 \ddt \left(\io u^p + \io |\na v|^{2q}\right)+\frac{c_1}2 \io |\na u^{\frac{p+m-1}{2}}|^2 +c_2 \io |\na v|^{2q} \leq c_5
\end{equation}
on $(0,\Tmax)$. Moreover, as a consequence of Lemma \ref{lem:intup} there is $c_6>0$ fulfilling 
\begin{equation}\label{eq:329}
 c_2\io u^p \leq \frac{c_1}2\io |\na u^{\frac{p+m-1}2}|^2  + c_6 
\end{equation}
and combining \eqref{eq:328} and \eqref{eq:329} shows that 
\[
 \ddt\left(\io u^p+\io |\na v|^{2q}\right) + c_2 \left(\io u^p + \io |\na v|^{2q} \right) \leq c_5+c_6
\]
for all $t\in (0,\Tmax)$. If we define $y(t):=\io u^p(\cdot,t)+\io |\na v(\cdot,t)|^{2q}$, $t\in (0,\Tmax)$, and $c_7:=c_5+c_6$, the inequality reads 
\[
 y'(t) + c_2 y(t) \leq c_7 \qquad \mbox{for all }t\in (0,\Tmax).
\]
Upon an ODE comparison, this entails boundedness of $y$ and hence \eqref{345}.
\end{proof}

A direct consequence is the following assertion on boundedness:
\begin{cor}\label{cor:unavbd}
 Let $\mu, \chi, \xi\ge 0$, $\delta>0$ and $m>2-\frac2n$.
 For any $p\in[1,\infty)$ there is $C>0$ such that for any function $D$ obeying \eqref{D12} and \eqref{D:nondeg}, the solution to \systemref\ fulfills
\[
 \norm[L^p(\Om)]{u(\cdot,t)}\leq C \qquad \mbox{and}\qquad \norm[L^p]{\na v(\cdot,t)}\leq C
\]
for any $t\in(0,\Tmax)$.
\end{cor}

Since $p=\infty$ is not covered by this corollary, we set out to improve the norms which can be controlled.
In the case of $v$ nothing more than a short application of $L^p-L^q$-estimates is needed:

\begin{lem}\label{lem:vbd}
 Let $\mu, \chi, \xi\ge 0$, $\delta>0$ and $m>2-\frac2n$.
 There is $C>0$ such that for any function $D$ obeying \eqref{D12} and \eqref{D:nondeg}, the solution to \systemref\ satisfies
\[
 \norm[W^{1,\infty}(\Om)]{v(\cdot,t)}\leq C
\]
for all $t\in [0,\Tmax)$.
\end{lem}
\begin{proof}
 We employ Corollary \ref{cor:unavbd} for some $p>n$ to gain a time-uniform bound on $\norm[L^p(\Om)]{u(\cdot,t)}$. Then Lemma \ref{lplqestimates} (iii), (iv) show that there is $c_1>0$ such that $\norm[L^\infty(\Om)]{\na v(\cdot,t)}\leq c_1$ and $\norm[L^\infty(\Om)]{v(\cdot,t)}\leq c_1$ for all $t\in(0,\Tmax)$.
\end{proof}

Turning Corollary \ref{cor:unavbd} into boundedness of $u$ is more difficult.
The lack of a bound on $\na w$ makes $\norm[L^\infty(\Om)]{u(t)}$ inaccessible for general boundedness results like \cite[Lemma A.1]{Tao-Winkler-jde}. We proceed somewhat similarly to \cite{tao2014boundedness}:
\begin{lem}\label{lem:ubd}
 Let $\mu, \chi, \xi\ge 0$, $\delta>0$ and $m>2-\frac2n$. There exists $C>0$ such that for any function $D$ satisfying \eqref{D12} and \eqref{D:nondeg}, the solution to \systemref\ fulfils
\[
 \norm[L^\infty(\Om)]{u(\cdot,t)} \leq C
\]
for all $t\in [0,\Tmax)$.
\end{lem}
\begin{proof}
From Lemma \ref{lem:vbd} we obtain $c_1>0$ such that $\norm[L^\infty(\Om)]{\na v(\cdot,t)}\leq c_1$ and $\norm[L^\infty(\Om)]{v(\cdot,t)}\leq c_1$ for all $t\in(0,\Tmax)$.

First let us fix $p_0>\frac32(m-1)$ so large that $\frac{p(p-1)}{(p+m-1)^2}\in(\frac12,\frac32)$ for all $p\in[p_0,\infty)$.

From Lemma \ref{l4} we see that
\begin{align*}
 \frac1p\ddt \io u^p + \frac{\delta(p-1)}{(p+m-1)^2}\io  |\nabla u^{\frac{p+m-1}2}|^2 \leq& \frac{\chi^2(p-1)}\delta \io u^{p-m+1}|\na v|^2 \\ &+ \xi\norm[L^\infty(\Om)]{w_0}\io u^p v + (\mu+\xi K)\io u^p-\mu\io u^{p+1}
\end{align*}
with $K$ as in \eqref{eq:defK} and for all $p\geq p_0, t\in(0,\Tmax)$. We may conclude
\begin{align}\label{eq:inequp}
 \ddt\io u^p+\io u^p + \frac\delta2 \io |\na u^{\frac{p+m-1}2}|^2 \leq& p^2 c_2 \io u^{p-m+1}+c_3p\io u^p-\mu p\io u^{p+1}\\
\leq& p^2 c_2 \io u^{p+m-1}+c_3p\io u^p-\mu p\io u^{p+1}+p^2c_2|\Om|\nn
\end{align}
for all $p\geq p_0$ and all $t\in (0,\Tmax)$, where
\[
 c_2=c_1^2\frac{\chi^2}\delta\quad\mbox{ and }\quad c_3=\xi\norm[L^\infty(\Om)]{w_0}c_1+\mu+\xi K+1 \quad \mbox{ and } \delta < \frac{\delta}2 \frac{4p(p-1)}{(p+m-1)^2}. 
\]
According to Young's inequality we may estimate
\[
 c_3p\io u^p\leq \mu p\io u^{p+1}+ \frac{p}{p+1}\left(\frac{p+1}p\mu\right)^{-p}c_3^{p+1}|\Om|.
\]
Furthermore, with $a=\frac n{n+2}$ we have $\frac12=(\frac12-\frac1n)a+\frac{1-a}1$, so that the Gagliardo-Nirenberg inequality gives $c_4, c_5>0$ such that by Young's inequality
\begin{align*}
 \norm[L^2(\Om)]{\phii}^2 \leq& c_4(\norm[L^2(\Om)]{\na\phii}^{\frac n{n+2}}\norm[L^1(\Om)]{\phii}^{\frac2{n+2}})^2+c_5\norm[L^1(\Om)]{\phii}^2\\
\leq& \frac{\delta}{2p^2 c_2}\norm[L^2(\Om)]{\na\phii}^2+\left(c_5+\frac{2}{n+2}(\frac{2n}{n+2}\frac{p^2c_2}{\delta})^{\frac{n}{2}}c_4^{\frac{n+2}{2}}\right)\norm[L^1(\Om)]{\phii}^2
\end{align*}
for all $\phii\in W^{1,2}(\Om)$.

If we apply this to $\phii=u^{\frac{p+m-1}2}$ and insert these estimates into \eqref{eq:inequp}, we arrive at
\begin{equation}\label{eq:ddtuppup}
 \ddt\io u^p+\io u^p \leq \left(\frac{c_3 p}{p+1}\right)^{p+1}\!\!\!\!\!\!\!\! \mu^{-p} |\Om|+p^2c_2\left(c_5+\frac{2}{n+2}(\frac{2n}{n+2}\frac{p^2c_2}{\delta})^{\frac{n}{2}}c_4^{\frac{n+2}{2}}\right)\left(\io u^{\frac{p+m-1}2}\right)^2+p^2c_2|\Om|
\end{equation}
for any $p\geq p_0$ and $t\in(0,\Tmax)$.

We then recursively define $p_k:=2p_{k-1}+1-m$ for $k\in\N$ and note that $p_0>\frac32(m-1)$ and $m\geq 1$ by an inductive argument lead to $p_0\geq \frac{p_k}{2^k}\geq(\frac12+2^{-k})(m-1)>0$ for all $k\in\N$ (at least if $m>1$; if $m=1$, $p_k=2^k p_0$), which obviously entails the existence of $c_6,c_7>0$ such that \begin{equation}\label{eq:pkest} c_6\leq \frac{p_k}{2^k}\leq c_7\end{equation} for any $k\in\N$.
Furthermore, we fix $T\in (0,\Tmax)$ and denote
\[
 M_k := \sup_{t\in(0,T)} \io u^{p_k}(t).
\]
From \eqref{eq:ddtuppup} we hence conclude
\begin{align*}
 \io u^{p_k} \leq& \io u_0^{p_k} \\
 +& \int_0^t e^{-(t-\tau)} \left(\left(\frac{c_3 p}{p+1}\right)^{p+1}\!\!\!\!\!\!\!\! \mu^{-p} |\Om|+\left(p^2c_2c_5+
 (\frac{2c_2c_4}{n+2})^{\frac{n+2}2} n^{\frac n2}\delta^{-\frac n2}p^{n+2}\right)M_{k-1} +p^2c_2|\Om| \right) d\tau,
\end{align*}
that is
\[
 M_k \leq \norm[L^{p_k}(\Om)]{u_0}^{p_k} + \mu^{-p_k}c_3^{p_k+1}|\Om| + p_k^2c_2|\Om| + c_8 p_k^{n+2} M_{k-1}^2
\]
for all $k\in\N$ with
\[
 c_8:=c_2c_5+(\frac{2c_2c_4}{n+2})^{\frac{n+2}2} n^{\frac n2}\delta^{-\frac n2}. 
\]
Either there is a subsequence $(p_{k_l})_l$ such that
\[
 c_8 p_{k_l}^{n+2} M_{k_l-1}^2 \leq \norm[L^{p_k}(\Om)]{u_0}^{p_k} + \mu^{-p_k}c_3^{p_k+1}|\Om| + p_k^2c_2|\Om|
\]
and hence
\begin{align*}
\norm[L^\infty(\Om)]{u}\leq& \limsup_{l\to\infty} M_{k_l}^{\frac1{p_{k_l}}}\\
 \leq& \limsup_{l\to\infty} \left(2\left[\norm[L^{p_k}(\Om)]{u_0}^{p_k} + \mu^{-p_k}c_3^{p_k+1}|\Om| + p_k^2c_2|\Om|\right]\right)^{\frac1{p_{k_l}}}\leq \norm[L^\infty(\Om)]{u_0} + \frac{c_3}\mu +1
\end{align*}
or for all sufficiently large $k$ we have $M_k\leq 2 c_8p_k^{n+2} M_{k-1}^2$ and hence
\[
 M_k\leq c_9p_k^{n+2}M_{k-1}^2\leq c_9c_7 (2^k)^{n+2} M_{k-1}^2
\]
for some $c_9>0$ and for all $k\in \N$ by \eqref{eq:pkest} and thus
\[
 M_k\leq a^k M_{k-1}^2 \qquad \mbox{ for all }k\in\N \qquad \mbox{with }a=\max\set{c_9c_7,1} 2^{n+2}.
\]
Induction leads to
\begin{equation}\label{eq:Mk}
 M_k\leq a^{k+\sum_{j=1}^{k-1} 2^j(k-j)} M_0^{2^k}.
\end{equation}
Since by another inductive argument we have $k+\sum_{j=1}^{k-1} 2^j (k-j) \leq 2^{k+1}$ for all $k\in\N$ and because $\frac{2^k}{p_k}$ is bounded, we hence obtain from \eqref{eq:Mk} that
\[
 M_k^{\frac1{p_k}}\leq a^{2\frac{2^k}{p_k}} M_0^{\frac{2^k}{p_k}} \leq (a^2 M_0)^{\frac1c_6}=:c_{10}.
\]
Taking $k\to \infty$, we finally arrive at
\[
 \sup_{t\in(0,T)} \norm[L^\infty(\Om)]{u(\cdot,t)} \leq c_{10}
\]
and due to the arbitrarity of $T$, this shows the claim.
\end{proof}

We are now in position to pass to the proof of Theorem 1.1.

\begin{proof}[Proof of Theorem 1.1]
According to Lemma \ref{local existence}, there exists a classical local-in-time solution $(u,v,w)$ to \systemref\ on $[0,\Tmax)$.
By Lemmata \ref{lem:ubd}, \ref{lem:vbd} and \ref{local existence}, $u$, $v$ and $w$ are uniformly bounded on $[0,\Tmax)$ and consequently \eqref{T} makes it possible to conclude $\Tmax=\infty$.
\end{proof}

\section{Degenerate diffusion and global weak solutions. Proof of Theorem \ref{Theorem2}}\label{s4}
The goal of this section is to prove Theorem \ref{Theorem2}.  In the absence of (\ref{D:nondeg}), the first equation of system \systemref\ may be degenerate at $u=0$ and we cannot hope for classical solutions. Therefore we introduce the definition of weak solution to \systemref\ which we shall pursue here.
\begin{defn}\label{def:weaksol}
Let $T\in(0, \infty]$ and $\Omega\subset\mathbb{R}^n$ be a bounded domain with smooth boundary.  A triple $(u, v, w)$ of nonnegative functions defined on $\Omega\times(0, T)$ is called a weak solution to \systemref\ on $[0, T)$ if\\
(i) $u\in L^2_{loc}([0, T);L^2(\Omega))$, $v\in L^2_{loc}([0, T);W^{1,2}(\Omega))$ and $w\in L^{2}_{loc}([0, T);W^{1,2}(\Omega))$;\\
(ii) $D(u)\nabla u\in L^1_{loc}([0, T);L^1(\Omega))$;\\
(iii) $(u, v, w)$ satisfies \systemref\ in the sense that for every $\varphi\in C_0^{\infty}(\Ombar\times[0, T))$,
\begin{eqnarray}
-\int_0^T\io u\varphi_t-\io u_0(x)\varphi(x, 0)dx&=&-\int_0^T\io D(u)\nabla u\cdot\nabla\varphi+\chi\int_0^T\io u\nabla v\cdot\nabla\varphi\nn\\
&&+\xi\int_0^T\io u\nabla w\cdot\nabla\varphi+\mu\int_0^T\io u\varphi-\mu\int_0^T\io  u^2\varphi\nonumber\\
&&-\mu\int_0^T\io uw\varphi;\label{eq:wsoldef_u}
\end{eqnarray}
holds as well as
\begin{equation}
-\int_0^T\io v\varphi_t-\io v_0(x)\varphi(x, 0)=-\int_0^T\io \nabla v\cdot\nabla\varphi-\int_0^T\io v\varphi+\int_0^T\io u\varphi \label{eq:wsoldef_v}
\end{equation}
and
\begin{equation}
-\int_0^T\io w\varphi_t-\io w_0(x)\varphi(x, 0)=-\int_0^T\io vw\varphi. \label{eq:wsoldef_w}
\end{equation}
In particular, if $T=\infty$ can be taken, then $(u,v,w)$ is called a global-in-time weak solution to $\systemref$.
\end{defn}
In order to obtain a weak solution to \systemref, we start with the approximate problem, for $\eps\in(0,1)$ given by
\begin{eqnarray}\label{*1}
\left\{\begin{array}{lll}
     \medskip
     u_{\varepsilon t}=\nabla\cdot(D_{\varepsilon}(\ue )\nabla \ue )-\chi\nabla\cdot(\ue \nabla v_{\varepsilon})-\xi\nabla\cdot(\ue \nabla w_{\varepsilon})+\mu \ue (1-\ue -w_{\varepsilon}),\ \ \ &x\in \Omega,\ t>0,\\
      \medskip
     v_{\varepsilon t}=\Delta v_{\varepsilon}-v_{\varepsilon}+\ue ,\ \ &x\in \Omega,\ t>0,\\
      \medskip
     w_{\varepsilon t}=-v_{\varepsilon}w_{\varepsilon},\ \ &x\in \Omega,\ t>0,\\
      \medskip
      \frac{\partial \ue }{\partial \nu}=\frac{\partial v_{\varepsilon}}{\partial \nu}=\frac{\partial w_{\varepsilon}}{\partial \nu}=0,\ \ &x\in\partial\Omega,\ t>0,\\
      \medskip
      \ue (x, 0)=u_{0}(x),\ v_{\varepsilon}(x, 0)=v_{0}(x),\ w_{\varepsilon}(x, 0)=w_{0}(x)\ \ &x\in\Omega
 \end{array}\right.
\end{eqnarray}
where $D_{\varepsilon}(s)=D(s+\varepsilon)$ for all $s\geq0$.

Since $D\in C^2([0,\infty))$ fulfils $D(s)\geq \delta s^{m-1}$ for all $s>0$, also for $D_\eps$ the estimate
\begin{equation}\label{eq:Depsest}
 D_\eps(s)=D(s+\eps)\geq \delta(s+\eps)^{m-1}\geq \delta s^{m-1}
\end{equation}
holds with the same values of $\delta$ and $m$, which, notably, are independent of $\eps$.
Furthermore,
\begin{equation}\label{eq:Depsnull}
 D_\eps(0)=D(\eps)>\delta \eps^{m-1} >0.
\end{equation}
Accordingly, Theorem \ref{Theorem1} and the lemmata from its proof become applicable so as to yield global classical bounded solutions $(\ue,\ve,\we)$ to \eqref{*1} and the following bounds, uniformly in $\eps>0$:
\begin{lem}\label{lemma4.1}
Let $n\in\{2, 3, 4\}, \chi, \xi, \mu>0$ and suppose that $m>2-\frac{2}{n}$ and $\delta>0$. Assume that the initial data $(u_0, v_0, w_0)$ satisfy $(\ref{eq:initdatacond})$ and the diffusion function $D$ fulfills (\ref{D12}).  Then for any $p>1$ there exists a constant $C>0$ such that
\begin{align}
 \norm[L^\infty(\Om)]{\ue(\cdot,t)}\leq& C, \label{40}\\
 \norm[W^{1,\infty}(\Om)]{\ve(\cdot,t)}\leq& C, \label{vinfty}\\
 \norm[L^\infty(\Om)]{\we(\cdot,t)}\leq& C, \label{winfty}\\
 \intnt\io \De(\ue)\ue^{p-2}|\na\ue|^2\leq& C(1+t),\label{41}\\
\intnt\io \ue^{m+p-3}|\na \ue|^2 \leq& C(1+t),\label{42}\\
 \io |\na\ve(\cdot,t)|^2\leq& C\label{43},\\
 \intnt\io |\na\we|^2\leq& C(1+t).\label{44}
\end{align}
 for all $t\in(0,\infty)$.
\end{lem}

\begin{proof} The bounds in \eqref{40}, \eqref{vinfty} and \eqref{winfty} are direct conclusions of Theorem \ref{Theorem1}. The estimates
\eqref{41}, \eqref{42} can be gained directly from Lemma \ref{l4} by integration with respect to time. From \eqref{vinfty} we furthermore know there exists a constant fulfilling (\ref{43}). 
Moreover, according to Lemma \ref{lem:bdDeltaw}, we have
\[
 \io |\na \we|^2 = -\io \we\Delta \we \leq \norm[L^\infty]{w_{0\eps}} \io \ve\we + (K+1)\io \we,
\]
so that the boundedness of $\intnT\io |\na\we|^2$ in \eqref{44} results from the bounds on $\ve$ and $\we$.
\end{proof}

We can use these bounds in a straightforward manner to extract convergent subsequences of $\ve, \we$:
\begin{lem}\label{lem:subseq_vw}
 There exist $v,w\in L^2_{loc}((0,\infty),W^{1,2}(\Om))\cap L^\infty(\Om\times(0,\infty))$ and a sequence $(\eps_k)_{k\in\N}\searrow 0$ such that
\begin{align}
 v_{\varepsilon}&\rightharpoonup v &&\mbox{in}\ \ L^2_{loc}((0, \infty),L^2(\Omega))\label{sh5},\\
 \ve& \wstarto v &&\mbox{in}\ \ L^\infty(\Om\times(0,\infty))\label{conv:velinfty}\\
 \na\ve& \wstarto \na v &&\mbox{in}\ \ L^\infty(\Om\times(0,\infty))\label{conv:navelinfty}\\
w_{\varepsilon}&\rightharpoonup w &&\mbox{in}\ \ L^2_{loc}((0, \infty),L^2(\Omega)\label{sh7},\\
\nabla w_{\varepsilon}&\rightharpoonup\ \nabla w\ &&\mbox{in}\ \ L^2_{loc}((0, \infty),L^2(\Omega)\label{sh8},\\
\we& \wstarto w &&\mbox{in}\ \ L^\infty(\Om\times(0,\infty))\label{conv:welinfty}
 \end{align}
 as $\eps=\eps_k\searrow 0$.
\end{lem}
\begin{proof}
 This directly follows from the estimates \eqref{vinfty}, \eqref{winfty}, \eqref{44}.
\end{proof}

Due to the nonlinearities involved, we require better convergence properties of $\ue$.
As preparation for an Aubin-Lions-argument, let us first ensure boundedness of the time-derivatives in a certain weak sense. 
The reasoning is similar as in \cite[Proof of Theorem 1.1]{tao2013locally} and \cite[Proof of Theorem 1.3]{wang2014quasilinear}.
\begin{lem}\label{Lemma4.2}
Suppose that $m>2-\frac2n$. Let $\theta>\max\set{1,\frac{m}2}$. Then there is $r>1$ such that for any $T>0$ there exists $C>0$ such that
\begin{equation}\label{eq:normlew2rstar}
 \norm[L^1((0,T);(W^{2,r}_0(\Om))^*)]{(\ue^\theta)_t}\leq C
\end{equation}
holds true for any $\eps\in(0,1)$.
\end{lem}
\begin{proof}
We choose $r>1$ so large that $W^{2,r}_0(\Omega)\hookrightarrow W^{1, \infty}(\Omega)$.
Using that $C_0^\infty(\Om)$ is dense in $W_0^{2,r}(\Om)$, we may express the norm as
\begin{equation}\label{eq:dualnorm}
 \norm[L^1((0,T);(W^{2,r}_0(\Om))^*)]{(\ue^\theta)_t}=\intnT\norm[(W^{2,r}_0(\Om))^*]{(\ue^\theta)_t} = \intnT \sup_{\zeta\in C_0^\infty(\Om), \norm[W^{2,r}(\Om)]{\zeta}\leq1} \io (\ue^\theta)_t \zeta.
\end{equation}
In order to estimate this integral, let $\zeta\in C_0^\infty(\Om)$ be such that $\norm[W^{2,r}(\Om)]{\zeta}\leq 1$ and hence $\norm[L^\infty]{\nabla\zeta}\leq c_1$ and $\norm[L^\infty(\Om)]{\zeta}\leq c_1$ with $c_1$ being the embedding constant of $W_0^{2,k}(\Omega)\hookrightarrow W^{1, \infty}(\Omega)$.
We then multiply the first equation in \eqref{*1} by $\ue^{\theta-1}\zeta$ and, upon an integration by parts, obtain
\begin{eqnarray}\label{te1}
\frac{1}{\theta}\io (\ue ^{\theta})_t\cdot\zeta
&=&-(\theta-1)\io D_{\varepsilon}(\ue )\ue ^{\theta-2}
|\nabla \ue |^2\zeta-\io D_{\varepsilon}(\ue )\ue ^{\theta-1}\nabla \ue \cdot\nabla\zeta\nonumber\\
&&+\chi(\theta-1)\io \ue ^{\theta-1}\nabla \ue \cdot\nabla v_{\varepsilon}\zeta+\chi\io \ue ^{\theta}\nabla v_{\varepsilon}\cdot\nabla\zeta\nonumber\\
&&+\xi(\theta-1)\io \ue ^{\theta-1}\nabla \ue \cdot\nabla w_{\varepsilon}\zeta+\xi\io \ue ^{\theta}\nabla w_{\varepsilon}\cdot\nabla\zeta\nonumber\\
&&+\mu\io \ue^{\theta}(1-u_{\varepsilon}-w_{\varepsilon})\zeta=: I_1+\ldots+I_7
\end{eqnarray}
on $(0,\Tmax)$. By \eqref{40} and \eqref{winfty}, there are $c_2,c_3>0$ such that $\norm[L^\infty(\Om\times(0,T))]{\ue}\leq c_2$ and $\norm[L^\infty(\Om\times(0,T))]{\we}\leq c_3$, so that on $(0,\Tmax)$
\begin{equation}\label{eq:I7}
 |I_7|\leq \mu \io \ue^\theta (1-\ue-\we) |\zeta| \leq \mu c_2^\theta\cdot (1+c_2+c_3)\cdot \norm[L^\infty(\Om)]{\zeta}|\Om| \leq \mu c_2^\theta (1+c_2+c_3) c_1|\Om|=:c_4.
\end{equation}
Also $I_4$ can be estimated rather easily. With the aid of \eqref{43}, we find $c_5>0$ such that $\io |\na \ve(\cdot,t)|^2\leq c_5$ for all $t\in(0,T)$ and hence
\begin{equation}\label{eq:I4}
 |I_4|\leq \chi c_2^\theta \io |\na \ve \na \zeta| \leq \chi c_2^\theta \left(\io |\na \ve|^2 \io |\na \zeta|^2 \right)^{\frac12}\leq \chi c_2^\theta c_5^{\frac12} c_1=:c_6 \qquad \mbox{ on }(0,\Tmax).
\end{equation}
We do not have as convenient bounds for $I_6$, but we can prepare an estimate of its integral by means of \eqref{44}. In order to do so, we use Young's inequality to see that with $c_7:=\xi c_2^\theta c_1\max\set{1,|\Om|}$ 
\begin{equation}\label{eq:I6}
|I_6|= \xi\io \ue ^{\theta}|\nabla \we\cdot\nabla\zeta|\leq \xi c_2^{\theta}c_1 \io (|\nabla \we|^2+1)\leq c_7+c_7\io |\na \we|^2 \qquad \mbox{ on }(0,\Tmax).
\end{equation}
We want to handle the terms containing $|\na\ue|^2$ by means of \eqref{42}. 
Therefore, let us fix $p>1$ such that $\theta\geq \max\set{p, \frac{p+m-1}2}$. Then by Young's inequality
\begin{align} \label{eq:I3}
|I_3|=&\chi(\theta-1)\bigg|\io \ue ^{\theta-1}\nabla \ue \cdot\nabla v_{\varepsilon}\zeta\bigg|\leq\left(\io \ue ^{2\theta-2}|\nabla \ue |^2+\io |\nabla v_{\varepsilon}|^2\right)\cdot\|\zeta\|_{L^{\infty}(\Omega)}\nonumber\\
\leq&\chi(\theta-1)\left(\|\ue \|_{L^{\infty}(\Omega)}^{2\theta-m-p+1}\cdot\io \ue ^{m+p-3}|\nabla \ue |^2+\io |\nabla v_{\varepsilon}|^2\right)\cdot\|\zeta\|_{L^{\infty}(\Omega)}\nn\\
\leq& \chi(\theta-1)\left(c_2^{2\theta-m-p+1}\io \ue ^{m+p-3}|\nabla \ue |^2+c_5\right)c_1 
 = c_8 \io \ue ^{m+p-3}|\nabla \ue |^2 +c_9
\end{align}
on $(0,\Tmax)$, where we have set $c_8=\chi(\theta-1)c_2^{2\theta-m-p+1}c_1$, $c_9=\chi(\theta-1)c_5c_1$. Again by Young's inequality, we obtain 
\begin{align}\label{eq:I5}
|I_5|=&\xi(\theta-1)\bigg|\io \ue ^{\theta-1}\nabla \ue \cdot\nabla w_{\varepsilon}\zeta\bigg|
\leq \xi(\theta-1)\left(\io \ue ^{2\theta-2}|\nabla \ue |^2+\io |\nabla w_{\varepsilon}|^2\right)\cdot\|\zeta\|_{L^{\infty}(\Omega)}\nonumber\\
\leq&\xi(\theta-1)\left(\|\ue \|_{L^{\infty}(\Omega\times(0, \infty))}^{2\theta+1-m-p}\io \ue ^{m+p-3}|\nabla \ue |^2+\io |\nabla w_{\varepsilon}|^2\right)\cdot\|\zeta\|_{L^{\infty}(\Omega)},\nonumber.\\
\leq&\xi(\theta-1)\left(c_2^{2\theta-m-p+1} \io \ue ^{m+p-3}|\nabla \ue |^2+\io|\na \we|^2\right)c_1\nn \\
\leq& c_{10} \io \ue ^{m+p-3}|\nabla \ue |^2 + c_{11} \io|\na \we|^2 \qquad \mbox{ on }(0,\Tmax), 
\end{align}
with the abbreviations $c_{10}=\xi(\theta-1)c_2^{2\theta-m-p+1}c_1$ and $c_{11}=\xi(\theta-1)c_1$.
For the remaining two terms, we note that $D_\eps(\ue)=D(\ue+\eps)\leq \max_{0\leq s\leq c_2+1} D(s) =:d$ and that, by our choice of $\theta$, $2\theta-p-m+1>0$, so that 
\begin{align}\label{eq:I2}
|I_2|=&\bigg|\io D_{\varepsilon}(\ue )\ue^{\theta-1}\nabla \ue \cdot\nabla\zeta\bigg| 
\leq \left(\io D_{\varepsilon}(\ue )\ue ^{p-2}|\nabla \ue |^2+\io D_{\varepsilon}(\ue )\ue ^{2\theta-p}|\nabla \ue |^2\right)\cdot\|\nabla\zeta\|_{L^{\infty}(\Omega)}\nn\\
\leq& \left(\io D_{\varepsilon}(\ue )\ue ^{p-2}|\nabla \ue |^2+dc_2^{2\theta-p-m+1} \io \ue ^{m+2-3}|\nabla \ue |^2\right)c_1
\end{align}
on $(0, \Tmax)$ and
\begin{eqnarray}\label{eq:I1}
|I_1|&=&(\theta-1)\bigg|\io D_{\varepsilon}(\ue )\ue ^{\theta-2}
|\nabla \ue |^2\zeta\bigg|\leq
(\theta-1)\|\ue \|_{L^{\infty}(\Om)}^{\theta-p}\left(\io D_{\varepsilon}(\ue )\ue ^{p-2}|\nabla \ue |^2\right)\|\zeta\|_{L^{\infty}(\Omega)}\nn\\
&\leq&c_{12} \io D_{\varepsilon}(\ue )\ue ^{p-2}|\nabla \ue |^2
\end{eqnarray}
for any $t\in (0,\Tmax)$, if we set $c_{12}=(\theta-1)c_2^{\theta-p}c_1$. 
Finally combining \eqref{eq:I7} -- \eqref{eq:I1} with \eqref{eq:dualnorm} and \eqref{te1}, we obtain
\begin{align*}
 \frac1\theta&\norm[L^1((0,T);(W^{2,r}_0(\Om))^*)]{(\ue^\theta)_t}\leq \intnT (I_1+\ldots+I_7)\\
 \qquad &\leq (c_{12}+c_1) \intnT\!\!\!\!\io \De(\ue)\ue^{p-2} |\na \ue|^2 + c_1dc_2^{2\theta-p-m+1} \intnT\!\!\!\!\io \ue^{m+2-3}|\na \ue|^2\\
 \qquad &+ (c_8+c_{10})\intnT\!\!\!\!\io \ue^{m+p-3}|\na\ue|^2+(c_{11}+c_7) \intnT\!\!\!\!\io |\na \we|^2 + (c_9+c_6+c_7+c_4)T.
%
\end{align*}
Here we can apply \eqref{41}, \eqref{42} twice (also for $p=2$) and \eqref{44}, so that we obtain an $\eps$-independent bound for the right-hand side and thus have proven \eqref{eq:normlew2rstar}.
\end{proof}

\begin{lem}\label{lem:subseq_u}
Let $m>2-\frac{2}{n}$.  Then there exists a subsequence $(\varepsilon_j)_{j\in\mathbb{N}}$ of the sequence provided by Lemma \ref{lem:subseq_vw} such that for all $T>0$
\begin{align}
\ue \rightarrow u  \qquad &\mathrm{a.e.}\ \mathrm{in}\ \ \Omega\times(0, T)\label{sh2},\\
\ue \rightarrow u \qquad &\mathrm{in}\ \ L^2((0, T),L^2(\Omega))\label{sh1},\\
D_{\varepsilon}(\ue )\nabla \ue \rightharpoonup D(u)\nabla u\ \quad&\mathrm{in}\ \ L^2((0, T),L^2(\Om))\label{sh9}.
\end{align}
\end{lem}
\begin{proof}
We fix $p>1$ and $\theta=\frac{m+p-1}2$. Given $T>0$, the bound from \eqref{42} then shows that there is $C>0$ such that $\intnT\io |\na \ue^\theta|^2\leq C$ for any $\eps\in(0,1)$ and together with the uniform bound from \eqref{40} we infer that $(\ue ^{\theta})_{\varepsilon\in(0, 1)}$ is bounded in $L^2(0, T;W^{1, 2}(\Omega))$. From Lemma \ref{Lemma4.2}, we know that $((\ue ^{\theta})_t)_{\varepsilon\in(0, 1)}$ is bounded in $L^1((0, T);(W_0^{k, 2}(\Omega))^*)$. Since $W^{1,2}(\Om)\hookrightarrow\hookrightarrow L^2(\Om)\hookrightarrow (W_0^{k,2}(\Om))^*$, we therefore conclude from the Aubin-Lions compactness lemma \cite[Cor. 4]{Simon} that $(\ue ^{\theta})_{\varepsilon\in(0, 1)}$ is a relatively compact subset of the space $L^2(0, T;L^2(\Omega))$, whence there exists a sequence of numbers $\varepsilon=\varepsilon_j\searrow0$ such that $\ue ^{\theta}\rightarrow u^{\theta}$ in $L^2(0, T;L^2(\Omega))$ along this sequence with some function $u^\theta$ in this space, which, along a subsequence, ensures 
\eqref{sh2}. Together with \eqref{40}, Lebesgue's dominated convergence theorem therefore asserts \eqref{sh1} as well.\\
The boundedness of $\ue$ as warranted by \eqref{40} entails the uniform $L^\infty$-boundedness of $\De(\ue)$, so that inserting $p=2$ into \eqref{41} already ensures boundedness of $\De(\ue)\na\ue$ in $L^2((0,T), L^2(\Om))$, and after extracting a further subsequence we may assume that $\De(\ue)\na\ue$ converges weakly in $L^2((0,T),L^2(\Om))$. In the interest of identification of the limit, we resort to a primitive of $\De$, say $G_\eps(x):=\int_{-\eps}^x \De(s)ds=\int_0^x D(s)ds$. We already have observed that $\na [G_\eps(\ue)]= \De(\ue)\na\ue$ is weakly convergent in $L^2((0,T),L^2(\Om))$. Moreover, by \eqref{sh2} $\ue+\eps\to u$ a.e. in $\Om\times(0,T)$ so that 
\[
 G_\eps(\ue(\cdot))=\int_0^{\ue(\cdot)+\eps} D(s) ds \to \int_0^{u(\cdot)} D(s) ds =: G(u(\cdot))
\]
 a.e. in $\Om\times (0,T)$ and thus, again as consequence of \eqref{40} and Lebesgue's dominated convergence theorem, in $L^2((0,T),L^2(\Om))$. Therefore the weak limit of $\na[G_\eps(\ue)]$ has to coincide with $\na G(u)=D(u)\na u$, so that we arrive at \eqref{sh9}.
\end{proof}

\begin{proof}[Proof of Theorem \ref{Theorem2}]
 The convergence properties asserted in Lemma \ref{lem:subseq_u} and Lemma \ref{lem:subseq_vw} are sufficient to pass to the limit in each of the integrals making up the corresponding weak formulation associated with \eqref{*1} so that $(u,v,w)$ is a weak solution to \systemref\ in the sense of Definition \ref{def:weaksol}. The boundedness of $(u,v,w)$ in the sense of \eqref{eq:defbdness} results from \eqref{sh2}, \eqref{conv:velinfty}, \eqref{conv:welinfty} and \eqref{conv:navelinfty}.
\end{proof}

\section{Acknowledgements}
The first author has been supported by China Scholarship Council (No.201406090072) and in part by National Natural Science Foundation of China (No.11171063).


\end{document}